\documentclass{amsart}

\pdfoutput=1

\usepackage{etex}
\usepackage{graphicx}
\usepackage{bbm}
\usepackage{latexsym,wasysym}
\usepackage{soul}
\usepackage[all]{xy}
\usepackage{booktabs}
\usepackage{array}
\usepackage{multirow}
\usepackage{mathrsfs}
\usepackage{enumitem}
\usepackage[backref=page, bookmarks=true, bookmarksopen=true,%
bookmarksdepth=3,bookmarksopenlevel=2,%
colorlinks=true,%
linkcolor=blue,%
citecolor=blue,%
filecolor=blue,%
menucolor=blue,%
urlcolor=blue]{hyperref}
\usepackage{amsmath,bm,amsthm,amssymb,amsbsy,amstext,amsopn,amsfonts,psfrag,color,float}
\usepackage{verbatim}
\usepackage{microtype} % improves formatting
\usepackage{times} % now deprecated in favor of newtxtext, but the latter not supported by arXiv as of 12/16
\usepackage{tikz}
\usetikzlibrary{matrix,shapes,arrows,calc,topaths,intersections,positioning,decorations.pathreplacing,decorations.pathmorphing,fit}
\usepackage{xcolor}
\usepackage{caption}
\let\oldtocsection=\tocsection
\let\oldtocsubsection=\tocsubsection
\let\oldtocsubsubsection=\tocsubsubsection

\renewcommand{\tocsection}[2]{\hspace{0em}\oldtocsection{#1}{#2}}
\renewcommand{\tocsubsection}[2]{\hspace{3em}\oldtocsubsection{#1}{#2}}
\renewcommand{\tocsubsubsection}[2]{\hspace{6em}\oldtocsubsubsection{#1}{#2}}

\usepackage{cleveref}
\Crefname{construction}{Construction}{Constructions}

\newtheorem{lemma}{Lemma}[section]
\newtheorem{proposition}[lemma]{Proposition}
\newtheorem{definition/proposition}[lemma]{Definition/Proposition}
\newtheorem{corollary}[lemma]{Corollary}
\newtheorem{theorem}[lemma]{Theorem}

\theoremstyle{definition}
\newtheorem{definition}[lemma]{Definition}
\newtheorem{remark}[lemma]{Remark}

\newtheorem{example}[lemma]{Example}

\newtheorem{question}[lemma]{Question}
\newtheorem{hyp}[lemma]{Hypothesis}

\newcommand{\natls}{{\mathbb N}}

\newcommand\AAA{{\mathcal A}}

\newcommand\BB{{\mathcal B}}
\newcommand\CC{{\mathcal C}}
\newcommand\DD{{\mathcal D}}
\newcommand\EE{{\mathcal E}}
\newcommand\FF{{\mathcal F}}
\newcommand\GG{{\mathcal G}}
\newcommand\HH{{\mathcal H}}

\newcommand\LL{{\mathcal L}}
\newcommand\MM{{\mathcal M}}

\newcommand\PP{{\mathcal P}}
\newcommand\QQ{{\mathcal Q}}

\newcommand\SSS{{\mathcal S}}
\newcommand\TT{{\mathcal T}}

\newcommand\VV{{\mathcal V}}

\newcommand\PMF{{\PP\kern-2pt\MM\FF}}

\newcommand\PML{{\PP\kern-2pt\MM\LL}}

\newcommand\Hyp{{\mathbb H}}

\newcommand\Z{{\mathbb Z}}
\newcommand\R{{\mathbb R}}

\author{Mahan Mj}
\address{School of Mathematics, Tata Institute of Fundamental Research. 1, Homi Bhabha Road, Mumbai-400005, India}

\email{mahan@math.tifr.res.in}

\author{Balarka Sen}
\address{School of Mathematics, Tata Institute of Fundamental Research. 1, Homi Bhabha Road, Mumbai-400005, India}

\email{balarka2000@gmail.com, balarka@math.tifr.res.in}

\title{Tight contact structures on hyperbolic homology 3-spheres}

\numberwithin{equation}{subsection}

\begin{document}
	

	
	\thanks{Both authors are supported in part by  the Department of Atomic Energy, Government of India, under project no.12-R\&D-TFR-5.01-0500, and by an endowment of the Infosys Foundation. MM was supported in part by the Fields Institute, Toronto during the special semester on "Randomness and Geometry", and by a DST JC Bose Fellowship}.

\begin{abstract} We produce a large class of hyperbolic  homology $3$-spheres  admitting  arbitrarily many distinct tight contact structures. We also
	produce a  sub-class admitting arbitrarily many distinct tight contact structures within the same homotopy class of oriented plane distributions. {As a corollary, we give a recipe to construct hyperbolic $L$-spaces admitting arbitrarily many distinct tight contact structures}.
	We also introduce a notion of geometric limits of contact structures compatible with geometric limits of hyperbolic manifolds {and} study the behavior of the 
	tight contact structures we construct under geometric limits.\end{abstract}

\maketitle

{\small \tableofcontents}

\section{Introduction}\label{sec-intro}

A contact structure on a $3$-manifold $M$ is a rank-$2$ distribution $\xi \subset TM$ such that $\xi$ is nowhere integrable. Contact structures on $3$-manifolds come in two distinct flavors: tight and overtwisted. A result due to Eliashberg \cite{yasha-89} states that overtwisted contact structures are completely classified by the homotopy class of their underlying rank-$2$ distributions. Consequently, all oriented closed $3$-manifolds admit overtwisted contact structures. In \cite[8.2.]{yasha-92}, Eliashberg asked: which $3$-manifolds admit tight contact structures? Etnyre and Honda \cite{honet-01} have demonstrated that the reducible manifold $P \# \overline{P}$ does not admit a tight contact structure, where $P$ denotes the Poincar\'e homology sphere. Nevertheless, the following question remains open:

\begin{question}[Question 8.2.2. in \cite{yasha-92} and Problem 4.142 in \cite{kirby-list}]\label{yashaq} Let $M$ be an irreducible 3-manifold. Does it admit a tight contact structure?\end{question}

The problem of {existence} of tight contact structures on lens spaces \cite{honda-1} and Seifert fibered spaces \cite{gompf-98, honda-2, listi-09} is completely settled. By the work of Agol \cite[Theorem 9.2]{agol-13} and Wise \cite{wise-21}, every closed hyperbolic $3$-manifold admits a finite-sheeted cover which is fibered over the circle. Combining this with the work of Eliashberg-Thurston \cite{elithu-98} on confoliations, one deduces that every closed hyperbolic $3$-manifold admits a finite-sheeted cover possessing a tight contact structure. 

However, not much is known beyond this with regard to Question \ref{yashaq} for the class of hyperbolic $3$-manifolds. In this direction, there is a {classification of tight contact structures for many surgeries} on the figure-eight knot, due to Conway-Min \cite{conmin-20}. Moreover, there are examples of hyperbolic 3-manifolds which do not admit fillable contact structures, {independently by Kaloti-Tosun \cite{kaltos-17} and Li-Liu \cite{liliu-19}}. {Also, Mark-Tosun \cite[Theorem 1.3]{martos-18} produce many examples of hyperbolic manifolds with tight contact structures, by verifying that the Heegaard-Floer contact invariant does not vanish for all positive rational surgeries on certain knots, {generalizing earlier works of Lisca-Stipsicz \cite{listi-04, listi-11}.}} For more examples of tight contact structures on hyperbolic $3$-manifolds, see \cite{hokama-03}, \cite{arimer-17}, \cite{etg-12}.

Gabai \cite{gabai-83} observed that any oriented closed irreducible $3$-manifold $M$ with $b_1(M) > 0$ admits a taut foliation. Using the theory of confoliations, one can therefore deduce that such manifolds also admit tight contact structures. Thus, the attention can be narrowed down to  hyperbolic rational homology $3$-spheres. The main theorem of this article provides examples of tight contact structures on a large class of integer homology $3$-spheres.

\begin{theorem}\label{thm-main} There exists an infinite family of
	sequences $\{M_{i,1},M_{i,2}, M_{i,3},\cdots \}$ of hyperbolic integer homology $3$-spheres, such that
	\begin{enumerate}[label=\normalfont(\arabic*)]
	\item $M_{i,n} \to M_{i,\infty}$, i.e.\ $M_{i,n}$ converges \emph{geometrically}
	to the finite volume complete hyperbolic manifold $M_{i,\infty}$,
	\item $M_{i,\infty}$, $M_{j,\infty}$ are distinct for $i\neq j$.
	\end{enumerate}
 such that the following holds: For any such sequence $\{M_{i,1},M_{i,2}, M_{i,3},\cdots \}$, there is a sequence of natural numbers $(a_n)_{n \geq 1}$ with $a_n \to \infty$ such that 
\begin{enumerate}[label=\normalfont(\arabic*)]
\item $M_{i,n}$ admits at least $a_n$ contact structures $\xi_1, \cdots, \xi_{a_n}$ that are pairwise non-contactomorphic,
\item $\xi_1, \cdots, \xi_{a_n}$ are all Weinstein-fillable.
\item $\xi_1, \cdots, \xi_{a_n}$ all belong to the same homotopy class of rank-$2$ distributions on $M_n$.
\end{enumerate}
\end{theorem}

Gromov \cite{gromov-85} and Eliashberg \cite{yasha-90} observed that Weinstein-fillable contact structures are tight. Colin-Giroux-Honda \cite{colgirhon-09} established that for any closed oriented atoroidal $3$-manifold (e.g.\., hyperbolic $3$-manifold) $M$, the number of tight contact structures that $M$ admits, hereby denoted as $\#\mathrm{Tight}(M)$, is finite. We immediately have the following corollary of Theorem \ref{thm-main}:

\begin{corollary}\label{cor-tightgrows}
	For any sequence $M_{i,n} \to M_{i,\infty}$ as in Theorem \ref{thm-main}, $\#\mathrm{Tight}({M_{i, n}}) \to \infty$ as $n \to \infty$. \end{corollary}
	
A word about the history of this problem is in order:
\begin{enumerate}
\item In \cite{lismat-97}, Lisca-Matic proved an analogous result for certain Seifert-fibered rational homology $3$-spheres, given by $1/n$-Dehn surgery on the trefoil knot.
\item {In \cite{listi-04}, Lisca-Stipsicz proved the existence of Seifert-fibered rational homology spheres $M_n$ carrying at least $n$ non-contactomorphic tight, nonfillable contact structures.}
\item In \cite[Theorem A.7]{bm24}, Barthelme, Mann and  Bowden prove the existence of  closed hyperbolic 3-manifolds $M_n$ with at least $n$ non-contactomorphic
Anosov contact structures. However, their examples all have positive first Betti number, i.e.,
$b_1(M_n) > 0$. In fact, the examples in  \cite[Theorem A.7]{bm24} satisfy $b_1(M_n) \to \infty$ as $n \to \infty$.
\end{enumerate}

{We also remark that Conclusion $(3)$ of Theorem~\ref{thm-main} indicates that the $h$-principle, which always holds for overtwisted contact structures \cite{yasha-89}, fails arbitrarily badly for tight contact structures on the manifolds that we construct.}\\

\noindent
\textbf{Outline of the paper:}\\
The examples in Theorem~\ref{thm-main} arise as $1/n$-Dehn surgeries on certain hyperbolic knots $K$ as $n \to \infty$. The knots $K$ in question satisfy the following:
\begin{enumerate}
\item They are closures $\widehat{\beta}$ of pseudo-Anosov braids ${\beta}$,
\item {The braid $\beta$ has} Haken genus greater than one (see Definition~\ref{def-hg}).
\item Suppose ${\beta}$ has $m$ strands. Let $c_+$ and $c_-$ respectively denote the number of positive and negative crossings of $\beta$. Then
$$c_+ - 2 c_- - m \geq 1.$$
\end{enumerate}
Conditions (1) and (2) above guarantee that the closure $\widehat{\beta}$ is, in fact, a hyperbolic knot. This is established in Section~\ref{sec-braid} and provides a reasonably elementary sufficient condition guaranteeing hyperbolicity of $\widehat{\beta}$. A more sophisticated sufficient condition may be found in \cite{ito}. Condition (3) ensures the existence of a Legendrian representative of $K$ with an appropriate framing
(see Section~\ref{subsec-cntstr}).  {In the beginning of Section~\ref{sec-pfmain}, we assume in addition that for the braid $\beta$ above,
\begin{itemize}
\item [(4)] $c_+ + c_- \equiv m + 1 \pmod{2}$.
\end{itemize}
This is only logically necessary for the third conclusion of Theorem~\ref{thm-main} regarding the fact that  the underlying rank-$2$ distributions of the contact structures are in the same homotopy class. This part of the argument is somewhat independent of the arguments needed for the rest of the conclusion of Theorem~\ref{thm-main}, and is discussed in Section~\ref{subsec-hprincfail}. We refer the reader to Remark~\ref{rmk-generic} and Section~\ref{subsec-mfds} for a discussion on sufficient conditions that guarantee conditions (1), (2), (3) and (4).}

While Section~\ref{sec-pfmain} is devoted to a proof of Theorem~\ref{thm-main}, we generalize it in two directions in Section~\ref{sec-furex} (see Theorem~\ref{thm-eglinkrat}):
\begin{enumerate}
\item First, we allow links $L$ in place of $K$. Such links $L$ {are required to be closures of braids $\beta$ satisfying Conditions (1) and (2) above. If $\beta_i$, $1 \leq i \leq \ell$ denotes the braids corresponding to the $i$-th cycle in the cycle decomposition of the permutation associated to $\beta$, then each $\widehat{\beta}_i$ is a knot and $L = \widehat{\beta} = \widehat{\beta}_1 \cup \widehat{\beta}_2 \cup \cdots \cup \widehat{\beta}_{\ell}$. We further require that each $\beta_i$ satisfy a strong form of Condition $(3)$ above. Namely, let $c_{i,+}, c_{i,-}$ denote the number of positive and negative crossings of $\beta_i$, let $m_i$ denote the number of strands of $\beta_i$ and let $d_{i,-}$ denote the number of negative crossings between strands of $\beta_i$ and strands of $\beta_j$ for all $j \neq i$, $1 \leq j \leq \ell$. Then we demand,
$$c_{i,+} - 2c_{i,-} - d_{i,-} - m_i \geq 1.$$}

\item The $1/n$-surgeries are replaced by positive $p/q$-surgeries (on each component of $L$). {This is done in two stages, with $p/q \in (0, 1)$ in Section \ref{subsec-surglink} and then with arbitrary $p/q > 0$ in Section \ref{subsec-surgknot}. The key observation in the $p/q \in (0, 1)$ case is that ${-q/p}$ has a \emph{terminating negative continued fraction} expansion $[a_0,\cdots, a_n]$ with each $a_i \leq -2$ (see Definition~\ref{def-negcont}). For the general $p/q > 0$ case, we effectively reduce to the $p/q \in (0, 1)$ case by iterated applications of the slam-dunk operation on multiple disjoint copies of the unknotted meridian, see Proposition \ref{prop-allposcont}, and Corollaries~\ref{cor-allposknot} and \ref{cor-allposlink}. }
\end{enumerate}

This allows us to construct {hyperbolic} rational homology spheres $M$ with arbitrarily many tight contact structures by rational Dehn filling on links $L$ as above (see Corollary~\ref{cor-conthyprat}). {Moreover, we also obtain a large class of hyperbolic knots and links, such that all positive rational surgeries on them admit tight (in fact, Weinstein-fillable) contact structures (see Corollary~\ref{cor-allposknot} and Corollary~\ref{cor-allposlink}).} {We use this result in Section~\ref{subsec-lspace} to give a recipe for producing hyperbolic $L$-spaces $M_n$ with $\#\mathrm{Tight}(M_n) \to \infty$ (see Proposition \ref{prop-lspace} and Proposition \ref{prop-lspacetight}, as well as Example \ref{eg-lspace}).}

Section~\ref{sec-geolim} provides a notion of geometric convergence of hyperbolic 3-manifolds \emph{equipped with contact structures}. Tripp \cite{tripp-06} studied contact structures on non-compact 3-manifolds. He showed that there exist uncountably many non-contactomorphic tight contact structures on an irreducible {open} 3-manifold with an end of positive genus. We use the technology of bypasses developed by Honda \cite{honda-1,honda-2} together with the notion of slope at infinity from  \cite{tripp-06} to show that there exist uncountably many non-contactomorphic tight contact structures on the complement of a hyperbolic knot $K$ where
\begin{enumerate}
\item $S^3 \setminus K$ is a geometric {limit} of 
closed hyperbolic 3-manifolds \emph{equipped with contact structures},
\item The knots $K$ satisfy  conditions (1), (2)  and  (3) above.
\end{enumerate} 
\noindent
See Theorem~\ref{thm-geomlim} and Corollary~\ref{cor-inflts}. \\

\begin{comment}

\item some surgeries are $L$-spaces so do not admit taut foliations. laminar?
\item investigate virtual overtwistedness

\end{comment}

\noindent
\textbf{Acknowledgments:}  We thank Marc Kegel for comments on an earlier draft, {and B\"ulent Tosun for pointing us to the references \cite{martos-18} and \cite{kaltos-17}.}
The second author thanks Dishant Pancholi and Mike Miller Eisemeier for useful conversations.

\section{Construction of hyperbolic pseudo-Anosov braids}\label{sec-braid}
\subsection{Braid closures}\label{sec-braidclosure}
Let $B_m$ denote the braid group on $m$ strands, and let $\beta \in B_m$ be a braid. Let $\widehat{\beta} \subset S^3$ denote the link obtained from taking the closure of $\beta$. Fix $m$ points on a disk, hereby denoted as $p_1, \cdots, p_m \in D^2$. Let $\mathrm{MCG}_{\partial}(D^2, m)$ denote the mapping class group of the disk with $p_1, \cdots, p_m$ as marked points. Then, $B_m \cong \mathrm{MCG}_{\partial}(D^2, m)$. Therefore, $\beta$ gives rise to a mapping class $\phi$. 
Let 
$$T(\phi) = (I \times D^2)/(x, 0) \sim (\phi(x), 1),$$
denote the mapping torus of $\phi$. Fix a base-point $z_0 \in \partial D^2$. The image of $I \times \{z_0\}$ under the quotient map $I \times D^2 \to T(\phi)$ gives an embedded curve $c \subset \partial T(\phi)$. Choose a genus $1$ Heegaard decomposition $S^3 = (S^1 \times D^2) \cup (D^2 \times S^1)$, where we suppose coordinates on either torus are chosen in a way that $S^1 \times \{z_0\}$ in the first solid torus bounds a meridian disk $D^2 \times \{z_0\}$ in the second solid torus. Let us also choose an identification $T(\phi) \cong S^1 \times D^2$ given by sending $c$ to $S^1 \times \{z_0\}$. Consider the composition,
$$D^2 \times I \to T(\phi) \cong S^1 \times D^2 \hookrightarrow S^3,$$
where the first map is the quotient map to the mapping torus, the middle identification is the one chosen above, and the final embedding is the inclusion of the first solid torus in the chosen Heegaard decomposition of $S^3$. The image of $\bigsqcup_i \{p_i\} \times [0, 1]$ under the above map defines the link $\widehat{\beta} \subset S^3$, contained within the solid torus $S^1 \times D^2$. We call $S^1 \times D^2$ the \emph{braid torus}.

We will say that $\beta$ is \emph{pseudo-Anosov} if the mapping class $\phi$ is pseudo-Anosov. In this case, the interior of $(S^1 \times D^2) \setminus \widehat{\beta}$ admits a the structure of a complete  cusped hyperbolic $3$-manifold by Thurston's work on hyperbolization of fibered $3$-manifolds \cite{thurston-86, otal-book}.

\begin{definition}\label{def-hg} Let $\beta \in B_m$ be a pseudo-Anosov braid. We define an \emph{essential surface for $\beta$} to be a properly embedded compact surface $\Sigma \subset (S^1 \times D^2) \setminus \widehat{\beta}$ with boundary $\partial \Sigma = \bigsqcup_i S^1 \times \{x_i\}$ (where $x_i \in \partial D^2$), such that $\Sigma$ is incompressible, $\partial$-nonparallel and cusp-nonparallel. We define the \emph{Haken genus} of $\beta$ as
$$h(\beta) := \min \{g(\Sigma) : \Sigma \subset (S^1 \times D^2) \setminus \widehat{\beta}\} \cup \{\infty\},$$
where $g(\Sigma)$ denotes the genus of $\Sigma$, and $\Sigma$ varies over all essential surfaces for $\beta$.\end{definition}

\noindent {\bf Essential surfaces with minimal intersection with a Heegaard torus.}\\
Next, let $\Sigma \subset S^3 \setminus \widehat{\beta}$  be an embedded incompressible surface.
Note that 
$$S^3 \setminus \widehat{\beta} = \left (S^1 \times D^2 \setminus  \widehat{\beta} \right ) \bigcup \left ( D^2 \times S^1 \right ).$$
We refer to this decomposition as the Heegaard decomposition of  $S^3 \setminus \widehat{\beta}$ and $\TT=\partial D^2 \times S^1$ as the \emph{toral Heegaard surface}, or simply the Heegaard torus. As in Definition~\ref{def-hg}, we refer to 
$(S^1 \times D^2) \setminus  \widehat{\beta}$ as the braid-torus. We also 
refer to 
$ D^2 \times S^1 $ as the \emph{solid Heegaard torus}.

Assume that $\Sigma$ intersects $\TT$ transversely and minimally, i.e.\ $\Sigma \cap \TT$
has the minimum number of components possible. 

If $S^3 \setminus \widehat{\beta}$  is irreducible, then no component $\sigma$ of  $\Sigma \cap \TT$ is homotopically trivial. Indeed, otherwise  $\sigma$ bounds disks $D_\Sigma$ and 
$D_\TT$ on $\Sigma , \TT$ respectively. By irreducibility of $S^3 \setminus \widehat{\beta}$,
$D_\Sigma\cup
D_\TT$ bounds a ball, and we can isotope $\TT$ to reduce the number of components of 
$\Sigma \cap \TT$, contradicting minimality of the intersection. Hence,
$\Sigma\cap \TT$ consists of finitely many disjoint
freely homotopic essential curves $\gamma_1, \cdots, \gamma_k$ in $\TT$. Let $m/n$ denote the slope
of each $\gamma_i$. Here, we use the convention that a curve has slope $m/n$ if
it represents $m-$times a longitude plus $n-$times a meridian of the  \emph{solid Heegaard torus}.

\begin{lemma}\label{lem-slopezero}
Suppose $k>0$. Then each $\gamma_i$ has slope zero.
\end{lemma}

\begin{proof}
If not, then $\Sigma\cap \TT=\gamma_1\cup \cdots \cup \gamma_n$ where each $\gamma_i$ has slope $m/n$ with $m \neq 0$. Hence, $\gamma_i$ is homotopically non-trivial in the solid Heegaard torus $D^2 \times S^1$. Since $\Sigma\cap  ( D^2 \times S^1 )$ is minimal, it follows that each
component of $\Sigma\cap  ( D^2 \times S^1 )$ is an annulus $\AAA$ where the boundary curves have slope $m/n$ on $\TT$. Since any two such curves bound an annulus on $\TT$ isotopic rel.\ boundary to $\AAA$, we can isotope $\Sigma$ so that
each such $\AAA$ is isotoped rel. boundary into $\TT$; and subsequently into the braid torus. Since $\Sigma\cap \TT$ is minimal, it follows that 
$\Sigma\cap \TT=\emptyset$, contradicting $k>0$.
\end{proof}

 The following is immediate from Lemma~\ref{lem-slopezero}.
\begin{corollary}\label{cor-classfn} Let $\Sigma$ be an embedded incompressible surface in an irreducible $M = S^3 \setminus \widehat{\beta}$   having minimal intersection with $\TT$. Then
\begin{enumerate}[label=\normalfont(\arabic*)]
\item either $\Sigma$ is contained in the braid torus,
\item or each $\gamma_i$ bounds a disk in the solid Heegaard torus.
\end{enumerate}
\end{corollary}

\subsection{An application of Thurston's hyperbolization}\label{sec-thurston}

Recall that a compact  3-manifold $M$ (possibly with boundary) is \emph{atoroidal} if any essential torus in $M$ is homotopic into $\partial M$. Also, $M$ is  \emph{Haken} if $M$ is irreducible and admits an embedded incompressible surface of positive genus.

\begin{remark}\label{rmk-atoroidal}
The above definition of atoroidal is sometimes referred to as \emph{geometrically atoroidal}, where the (homotopically) essential torus is, by definition, embedded. However, there are a few examples of irreducible geometrically atoroidal  3-manifolds $M$ that satisfy $\Z^2 \subset \pi_1(M)$, where the  $\Z^2$ is carried by an immersed torus not homotopic into $\partial M$, i.e.\ $M$ is \emph{not homotopically atoroidal}. We discuss the relevant examples now. In the following, we shall be interested in link complements, where $\partial M$ consists of tori. If $M$ is atoroidal, this forces $M$ to be a Seifert-fibered space over a base-orbifold of type
$(p,q,\infty)$ with $p, q \in \natls \cup \{\infty\}$, and $p \leq q$. If $p=q=\infty$, $M$ is necessarily
homeomorphic to $S_{0,3} \times S^1$, where $S_{0,3}$ is a 3-holed sphere. In this case, $M$ is the complement of a trivial 3-component link. 

If $q =\infty$, and
$p \in \natls$, then $\pi_1(M)$ is $\Z \times (\Z\ast\Z/p)$ .   $M$ is the complement in  $S^3$ of a 2-component link $L$ of the form $L=K_1 \cup K_2$, where $K_1$ is the core curve of a solid torus $\mathbb T$
 embedded in a standard way in $S^3$, and $K_2$ is a $(p_1,p)$ torus knot on the boundary $\partial \mathbb T$ (with any $p_1 \geq 1$). We shall refer to such links $L$ as \emph{2-component $T^2 \times I-$links}, as the components of $L$ may be assumed to lie on the boundary components 
 of $T^2 \times \{0,1\}$ of $T^2 \times I$.

Finally, let $p \leq q < \infty$ and $(p,q)=1$. Then $M$ is the complement of a $(p,q)-$torus knot. 
It is indeed possible that such an $M$ can be the complement in  $S^3$ of the closure of a pseudo-Anosov braid. In fact, consider the braid group $B_3$ for braids on 3 strands, with standard generators
$\sigma_1, \sigma_2$. The braid $\beta=\sigma_1^3\sigma_2^{-1}$ satisfies
\begin{enumerate}
\item $\beta$ is pseudo-Anosov. Note that this means that its complement in the braid torus is hyperbolic.
\item $\widehat{\beta}$ is a  $(p,q)-$torus knot in $S^3$. Note that this implies that its complement in $S^3$ is \emph{not} hyperbolic.
\end{enumerate}
This unusual phenomenon accounts for the hypothesis in Corollary \ref{cor-pahakengenus} excluding such examples. 

More generally, if $dp_1=p \leq q =dq_1 < \infty$ and $(p_1,q_1)=1$, then $M$ is the complement of a link $L$ consisting of $d$ parallel copies of a  $(p_1,q_1)-$torus knot in $S^3$. We shall refer to  $M$ in this case as a \emph{ $(p,q)-$torus link}. But in this case, $L$ is a satellite of a $(p_1,q_1)-$torus knot; in fact it consists of $d$ parallel copies a  $(p_1,q_1)-$torus knot $K$. Let $\nu(K)$ denote a solid torus neighborhood of $K$.
Then $\partial  \nu(K)$ is an embedded incompressible torus in the complement of $L$.
\end{remark}

We start with Thurston's geometrization theorem:

\begin{theorem}\cite{otal-geom,kapovich-book}\label{thm-thurstonmonster}
Let $M$ be a compact atoroidal Haken 3-manifold (possibly with boundary). Then the interior of $M$ admits a complete hyperbolic structure.
\end{theorem}

\begin{corollary}\label{cor-pahakengenus}
	Let $\beta \in B_m$ be any pseudo-Anosov braid with Haken genus $h(\beta) > 1$. Then the closure $\widehat{\beta} \subset S^3$ is a hyperbolic link, provided
$\widehat{\beta}$ is not one of the following (see Remark~\ref{rmk-atoroidal}):
\begin{enumerate}[label=\normalfont(\arabic*)]
\item a torus knot,
\item a 2-component $T^2 \times I-$link,
\item a trivial 3-component link.
\end{enumerate}  \end{corollary}

\begin{proof}
	By the Definition of Haken genus, Remark~\ref{rmk-atoroidal}, and Theorem~\ref{thm-thurstonmonster}, it suffices 
	to show that $S^3 \setminus \widehat{\beta}$ is irreducible and atoroidal. Recall that $S^3 \setminus \widehat{\beta} = (S^1 \times D^2 \setminus  \widehat{\beta}) \bigcup ( D^2 \times S^1 )$ is a Heegaard decomposition of  $S^3 \setminus \widehat{\beta}$ and $\TT=\partial D^2 \times S^1$ is the toral Heegaard surface.
	
	If $S^3 \setminus \widehat{\beta}$ is reducible, then there exists an embedded essential sphere $\Sigma_0$ in $S^3 \setminus \widehat{\beta}$. We assume that $\Sigma_0$ intersects
	$\TT$ minimally. In particular, $\Sigma_0\cap \TT$ consists of finitely many disjoint
	freely homotopic essential curves in $\TT$. Further,  $ (S^1 \times D^2 \setminus  \widehat{\beta}) \cap \Sigma_0$ must then consist of finitely many annuli. But the existence of any such annulus would force the automorphism $\phi$ associated to $ \widehat{\beta}$ to be reducible, violating the hypothesis that $\phi$ is pseudo-Anosov.
	
	Next, if $S^3 \setminus \widehat{\beta}$ is not atoroidal, then there exists an embedded essential torus $\Sigma_1$ in $S^3 \setminus \widehat{\beta}$ that is not parallel to a boundary component. We assume that $\Sigma_1$ intersects
	$\TT$ minimally so that $\Sigma_1\cap \TT$ consists of finitely many disjoint
	freely homotopic essential curves in $\TT$. Further,  $ ( D^2 \times S^1) \cap \Sigma_1$
	must then 
	\begin{enumerate}
	\item either consist of finitely many embedded essential annuli,
	\item or finitely many embedded meridinal disks.
	\end{enumerate} 
	In the former case, each of these annuli can be homotoped into 
	$\TT$ and subsequently into $(S^1 \times D^2 \setminus  \widehat{\beta}) $. Thus,
	minimality of $\Sigma_1\cap \TT$ forces this intersection to be empty. Finally, 
	hyperbolicity of $(S^1 \times D^2 \setminus  \widehat{\beta}) $ forces $\Sigma_1$ to be homotopic to a boundary component of $(S^1 \times D^2 \setminus  \widehat{\beta}) $.
	If this boundary component is $\TT$, then it is compressible in $S^3 \setminus \widehat{\beta}$ violating essentiality of $\Sigma_1$. Else it is homotopic to a boundary component of a neighborhood of $ \widehat{\beta}$ in $S^3$. But this contradicts the choice
	of $\Sigma_1$ which was chosen not to be parallel to a boundary component of $S^3 \setminus \widehat{\beta}$.
	In the latter case, $h(\beta)=1$. 
	
	Note that the above argument also rules out the case that $\widehat{\beta}$ is a $(p,q)-$torus link in the sense of Remark~\ref{rmk-atoroidal}. Indeed $(p,q)>1$ forces the existence of an incompressible $\Sigma_1$ (with boundary) of genus equal to one.
	
The above  facts force $S^3 \setminus \widehat{\beta}$ to be irreducible and atoroidal. 
	Theorem~\ref{thm-thurstonmonster} now furnishes the conclusion.
 \end{proof}

 \subsection{Genus one surfaces}\label{sec-g1} This subsection is aimed at understanding understanding genus one essential surfaces $\Sigma_1$ in the braid torus.
 We assume further that ${\beta}$ is pseudo-Anosov.
 By Corollary~\ref{cor-classfn}, either such a $\Sigma_1$ is embedded in the braid torus or each component of $\Sigma_1\cap \TT$ bounds a disk in the solid Heegaard torus. Since $\beta$ is pseudo-Anosov, the first alternative forces $\Sigma_1$ to be parallel to a boundary component of the braid torus, i.e.\ $\Sigma_1$ is either $\partial-$parallel
 or cusp-parallel in the sense of Definition~\ref{def-hg}. We therefore assume henceforth that each component of $\Sigma_1\cap \TT$ bounds a disk in the solid Heegaard torus. In particular, the genus of $\Sigma_1$ equals that of
 $\Sigma_1 \cap (S^1 \times D^2 \setminus \widehat{\beta})$. 
 
 Next, we cut the braid torus $(S^1 \times D^2) \setminus \widehat{\beta}$ along 
$ (\{p\} \times D^2 ) \setminus \widehat{\beta}$ for some $p \in S^1$ so as to get
the braid complement $\CC=([0,1]\times D^2) \setminus {\beta}$. Further, from the above description of $\Sigma_1\cap \TT$, we can assume that
\begin{enumerate}
	\item $\Sigma_1$ is transverse to $ (\{p\} \times D^2) \setminus \widehat{\beta}$ in the braid torus $(S^1 \times D^2) \setminus \widehat{\beta}$. 
\item Let $\Sigma_1^c$ denote $\Sigma_1 \cap (S^1 \times D^2) \setminus \widehat{\beta}$ cut open along $ (\{p\} \times D^2) \setminus \widehat{\beta}$ (and then completed in $\CC$). Then $\Sigma_1^c$  intersects the \emph{vertical  boundary} $\partial_v \CC= [0,1]\times \partial  D^2 $ in vertical lines of the form
$[0,1]\times \{z_i\}$, $i=1,\cdots, 2l$,
\item $\Sigma_1^c$  intersects the \emph{horizontal  boundary} $\partial_h \CC=\{0,1\}\times  D^2 $ in
finitely many pairs of disjoint chords $\alpha_1 \times \{0,1\}, \cdots, \alpha_l 
\times \{0,1\}$ (the same for the top and bottom components of $\partial_h \CC$, where $\alpha_1, \cdots, \alpha_l 
$ are chords in the disk avoiding the initial points of the braid $\beta$.
\item The standard height function $H: \CC \to [0,1]$ given by the projection onto the first coordinate of $\CC$ is a Morse function on $\Sigma_1^c$.
\end{enumerate}

Let $\SSS_i:=\alpha_i \times [0,1] \subset D^2 \times  [0,1] $ denote \emph{vertical rectangles} corresponding to chords $\alpha_i$. Then 
$\partial \Sigma_1^c=\bigcup_i \partial \SSS_i$, i.e.\ the boundary 
of $\Sigma_1^c$ consists of the disjoint union of the boundary of the vertical rectangles $\SSS_i$. Further, there exist 
\begin{enumerate}
\item finitely many disks $\DD_1, \cdots, \DD_j$ contained in the interior of $\bigcup_i  \SSS_i$, and
\item finitely many orientable surfaces $\AAA_1, \cdots, \AAA_k$, where each
$\AAA_i$ is homeomorphic to a genus $g$ surface, $g\geq 0$ with exactly \emph{two}
boundary components. We refer to the $\AAA_i$ as \emph{generalized annuli}.
\end{enumerate}
such that 
$$\Sigma = \large(\cup_i  \SSS_i \setminus (\cup_r \DD_r) \large) \bigcup (\cup_s \AAA_s),$$
i.e.\ $\Sigma_1^c$ is obtained from $\bigcup_i  \SSS_i$ by removing the interiors of
the disks $\DD_r$ and attaching the generalized annuli $\AAA_s$ along the boundaries of the
removed disks. We construct an auxiliary (undirected) graph $\GG:= \GG(\SSS,\AAA)$
as follows:
\begin{enumerate}
\item the vertex set $\VV(\GG)$ equals $\{\SSS_i\}$,
\item the edge set $\EE(\GG)$ equals $\{\AAA_s\}$,
\item the edge $\AAA_s$ connects vertices $\SSS_a, \SSS_b$ if the boundary circles
of $\AAA_s$ lie on $\SSS_a, \SSS_b$. (Note that we allow $\SSS_a= \SSS_b$.)
\end{enumerate}
\begin{lemma}\label{lem-g1} Let $\Sigma_1$ be an essential genus $1$ surface in a braid link. Then, with notation as above, the generalized annuli $\AAA_s$ used in constructing $\Sigma_1^c$  all have $g=0$, i.e.\ each $\AAA_s$ is an annulus. Further, $\GG$ has exactly one edge-loop, 
	i.e.\ $\pi_1(\GG)=\Z$.
 \end{lemma}

\begin{proof} First, observe that $\GG$ encodes a graph of spaces description of 
	$\Sigma_1^c$, and hence a graph of groups 
	description of 
	$\pi_1(\Sigma_1^c)$, where
	\begin{enumerate}
		\item the vertex groups are $\pi_1(\SSS_i \setminus (\cup_r \DD_r))$,
	\item the edge groups are $\pi_1(\AAA_i)$.
	\end{enumerate} 
	Since $\Sigma_1^c$ has genus one, we can assume without loss of generality that each $\AAA_i$ is an annulus, so that 
	$\pi_1(\AAA_i)=\Z$. (Note that the only other possibility necessitates that $\AAA_1$ is  a twice-punctured torus connecting two squares $\SSS_1, \SSS_2$. This  is equivalent to having two disjoint annuli between $\SSS_1$ and $\SSS_2$. The associated graph of spaces has two vertices corresponding to $\SSS_1$ and $\SSS_2$ and two edges between them. We assume henceforth  that 
	such a twice-punctured torus is replaced by two annuli.)
	
	Next, we  glue the top and bottom of $\CC$ together to get back the braid torus. Gluing the horizontal boundary chords of $\Sigma_1^c$ back together, we also get back $\Sigma_1 \cap (S^1 \times D^2 \setminus \widehat{\beta})$.
	Note that each $\SSS_i \setminus (\cup_r \DD_r)$ now glues up to give an annulus
$\widehat{\SSS_i} \setminus (\cup_r \DD_r)$ 	with holes corresponding to $\DD_r$.
	Thus, $\Sigma_1 \cap (S^1 \times D^2 \setminus \widehat{\beta})$ admits a graph of spaces description, with the same underlying graph $\GG$, where
	\begin{enumerate}
	\item the vertex spaces are  $\widehat{\SSS_i} \setminus (\cup_r \DD_r)$,
	\item the edge spaces are the annuli  $\AAA_i$.
	\end{enumerate} 
	
	We next note that the genus of $\Sigma_1 \cap (S^1 \times D^2 \setminus \widehat{\beta})$ equals the genus of $\GG$, where the genus of $\GG$ equals the rank of its (free) fundamental group.
	Since the genus of $\Sigma_1 \cap (S^1 \times D^2 \setminus \widehat{\beta})$
	equals one, the conclusion follows.
\end{proof}

\subsection{Increasing genus}\label{sec-incgenus}
Theorem~\ref{thm-incgenus} below is the main output of this section.
We shall need some basic terminology. Let $\CC=(I \times D^2) \setminus \beta$ denote the braid complement as before.

\begin{definition}\label{def-saddle}
A surface $\FF \subset \CC$ is called
an \emph{elementary saddle} if there exist 4  cyclically ordered distinct points
$p_1,p_2,p_3,p_4 \in \partial D^2$, and a quadrilateral with sides $[p_1,p_2]$, $[p_2,p_3]$, $[p_3,p_4]$, $[p_4,p_1]
\subset D^2$ such that
\begin{enumerate}
	\item $\FF$ is homeomorphic to a closed disk.
	\item the boundary $\partial \FF$ consists of 
	\begin{itemize}
	\item the vertical arcs $\cup_i I \times \{p_i\} $,
	\item the horizontal arcs $\{0\} \times [p_1,p_2], \{0\} \times [p_3,p_4]$ on
	$\{0\} \times D^2$,
	\item the horizontal arcs $\{1\} \times [p_2,p_3]$, $\{1\} \times [p_4,p_1]$ on
	$\{1\} \times D^2$.
	\end{itemize}
	\item the height function  $H: \CC \to [0,1]$ given by projection to the first coordinate is a Morse function with a single saddle singularity.
\end{enumerate}

A surface $\FF \subset \CC$ is called
a \emph{trivial rectangle} if there exists a pair of distinct points $p,q$ and an arc $[p,q] \subset D^2$ such that $\FF$ is isotopic rel.\ boundary to $I \times [p,q]$.
\end{definition} 

\begin{figure}[h]
	\centering
	\includegraphics[scale=0.45]{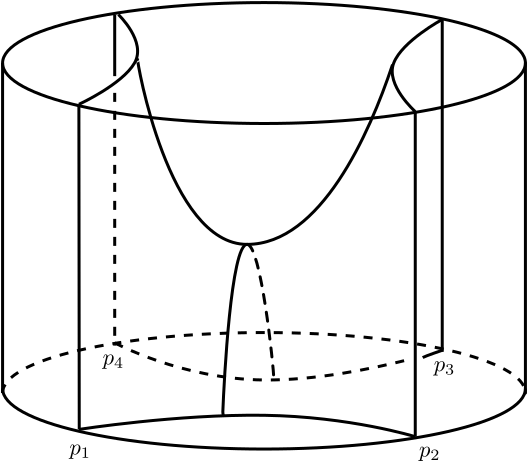}
	\caption{An elementary saddle}
	\label{fig-saddle}
\end{figure}

\begin{theorem}\label{thm-incgenus} For any pseudo-Anosov braid $\beta$, $h(\beta^k) > 1$ for $k>4$.\end{theorem}

\begin{proof} We first note that the second paragraph of the proof of Corollary \ref{cor-pahakengenus} in fact shows that $M=S^3 \setminus \widehat{\beta}$ is 
	irreducible. (The hypothesis $h(\beta) >1$ in Corollary \ref{cor-pahakengenus} was not used for the proof of this fact.)
Hence $h(\beta) \neq  0$. 
	
We observe that the complement $M_k$ in $S^3$ of the closure $\widehat{\beta^k}$ is obtained as follows. Let $M=S^3\setminus \widehat{\beta}$ as before. Let $\HH$ denote the solid Heegaard torus, and let $c \subset \HH$ denote the core curve of $\HH$. Then  $M_k$ is a $k-$fold branched cover of $M$ branched over $c$.

For concreteness, we show first that an essential embedded torus in $M$ pulls back to an essential surface of genus greater than one in $M_k$. (Strictly speaking, the
argument in this paragraph is not logically necessary for the rest of the proof, but it provides an easy first case.)
 Let $\Sigma_1$ denote the essential genus one surface in  $M$. 
Recall the structure of $\Sigma_1$ by Lemma~\ref{lem-g1}. Recall also that $\Sigma_1$ intersects the solid Heegaard torus
of $M$ in finitely many disks 	$\DD_1, \cdots, \DD_j$.
Let $\Sigma_1^k$ denote its pullback to $M_k$. Then $\Sigma_1^k$ intersects the solid Heegaard torus
of $M_k$ in finitely many disks 	$\DD_1', \cdots, \DD_j'$, where $\DD_r'$ is a $k-$fold branched cover of $\DD_r$ over the origin. Then $\Sigma_1^k$ has genus greater than one.
If possible, suppose that $\Sigma_1^k$ is compressible, and $\gamma \subset \Sigma_1^k$ is an essential closed curve that is contractible in $M_k$. Then 
 $\gamma$ can be isotoped slightly to be disjoint from the branch locus, forcing
 the image of $\gamma$ in $M$ to  be contractible. This contradicts the fact that
 $\Sigma_1$ is essential.
 
 It remains to show that there is no other essential surface in $M_k$ of genus
 zero or one. If $\Sigma_0$ is an essential sphere in $M_k$, its image, under the branched cover map $\Pi$ to $M$, would give an immersed essential sphere 
 $\Pi(\Sigma_0)$ in $M$. Hence by the
 sphere theorem, $M$ contains an essential embedded sphere, forcing 
 $h(\beta)=0$ and contradicting that $\beta$ is pseudo-Anosov as in the proof of
 Corollary \ref{cor-pahakengenus}.
 
 If $\Sigma_0'$ is an essential torus in $M_k$, its image  $\Pi(\Sigma_0')$ is an immersed essential torus in $M$. Let $\BB = (S^1 \times D^2) \setminus \widehat{\beta} \subset M$ denote the complement of the braid $\beta$ within the braid torus.   Let $\BB_k$ denote its lift
  to $M_k$,
 so that $\BB_k$ is a  $k-$fold regular cover of $\BB$. Note that $\BB_k$ is the complement of the braid $\beta^k$ within the braid torus of $\beta^k$. {We assume $\Sigma_0'$ intersects the boundary of the braid torus $\partial \BB_k$ (i.e., the Heegaard torus) transversely and minimally. By Corollary \ref{cor-classfn}, either $\Sigma_0'$ is contained entirely inside the braid torus or intersects the Heegaard torus in a collection of circles which bound meridinal disks in the solid Heegaard torus. The first scenario is impossible, as $\mathcal{B}_k$ is hyperbolic and therefore atoroidal, since $\beta$ is pseudo-Anosov.}

 Let $\CC_k=([0,k] \times D^2) \setminus \widehat{\beta^k}$ denote $\BB_k$ cut open along a disk. Let 
  $H_k: \CC_k \to [0,k]$ denote the associated height function. {Let 
  $\Sigma^{cc}_0 \subset \CC_k$  denote $\Sigma_0' \cap \BB_k$ cut along the disk}. We assume that 
 %Since $h(\beta)=1$, $M$ cannot be a `small' Seifert fibered space with base orbifold $(p,q,\infty)$, with $p, q \in \natls \cup \{\infty\}$. These Seifert fibered spaces gave rise to the exceptions in Corollary~\ref{cor-pahakengenus}. 
 \begin{enumerate}
 \item $H_k$ is a Morse function when restricted to $\Sigma^{cc}_0$,
 \item  the total number of critical points with respect to $H_k$ is the minimum possible in the isotopy class of $\Sigma^{cc}_0$ rel. boundary.  In particular,
 $H_k$ has no maxima or minima.
 \end{enumerate}
 Since $\BB_k$ is a  $k-$fold regular cover of $\BB$, we assume that $\CC_k$ is
 a vertical concatenation of $k$ copies of $\CC$. Let $\QQ_i:=(D^2 \times  [i-1,i])\setminus \beta$ denote the $i-$th block of $\CC$ in $\CC_k$, $i=1,\cdots,k$.
Let $\FF_i = \Sigma^{cc}_0 \cap \QQ_i$. 
By the discussion in
 Section~\ref{sec-g1} 
  we can assume that there exists $l >0$ and points $p_1, \cdots, p_{2l}$ such that 
  \begin{enumerate}
  \item the vertical boundary of $\FF_i$ equals  $[i-1,i] \times \{p_1, \cdots, p_{2l}\} $
  for all $i$,
  \item the horizontal boundary of $\FF_i$ restricted to $\{i\} \times D^2$ consists of $l$ disjoint arcs $\alpha_{i1}, \cdots, \alpha_{il}$ in  $\{i\} \times D^2$ pairing  points in $\{i\} \times \{p_1, \cdots, p_{2l}\} $.
  \end{enumerate}
  
  We next observe that $\FF_i$ cannot contain a trivial rectangle, else the trivial rectangle would separate strands of the braid $\beta$ into two non-empty clusters,
  forcing $\widehat{\beta}$ to be a split  link, violating the hypothesis that 
 $\beta$ is a  pseudo-Anosov braid. 
 Hence,   $\alpha_{i1}, \cdots, \alpha_{il}$ contribute at least $l/2$ elementary saddles in $\FF_i$, i.e.\ $\FF_i$ has at least $l/2$ critical
 points of index $1$. Recall that $H_k$ restricted to $\FF_i$ has no maxima
 or minima.
 
 Hence, the total number of index $1$ critical points on $\Sigma^{cc}_0$ is at least
  $k \cdot l/2$. {Hence, gluing $\Sigma^{cc}_0$ back along its top and bottom boundaries, we see $\Sigma_0' \cap \mathcal{B}_k$ has a circle-valued Morse function} with
  at least
  $k \cdot l/2$ index $1$ critical points, and no index $0$ or index $2$ critical points. Moreover, the circle-valued Morse function is identity restricted to the $2l$ boundary circles of $\Sigma_0' \cap \mathcal{B}_k$. Therefore, by the Poincare-Hopf index theorem,
$$\chi(\Sigma_0' \cap \mathcal{B}_k) \leq -k \cdot \frac{l}{2}$$
Note that $\Sigma_0'$ is obtained from $\Sigma_0' \cap \mathcal{B}_k$ by capping off the $2l$ boundary circles by disks. Thus, $\chi(\Sigma_0') = \chi(\Sigma_0' \cap \mathcal{B}_k) + 2l$. Combining with the above, we get,
$$\chi(\Sigma_0') \leq 2l - k \cdot \frac{l}{2}$$
Hence, for $k>4$, we must have $\chi(\Sigma_0') <0$. {Therefore, $\Sigma_0'$ has genus greater than $1$. Hence $\beta^k$ has Haken genus greater than $1$. This proves the theorem.}
\end{proof}

\begin{remark}\label{rmk-guaranteeingpahyp}
We summarize some conditions to ensure that a braid has hyperbolic braid closure.
From Corollary~\ref{cor-pahakengenus} and Theorem~\ref{thm-incgenus} we have the following. Let $\beta$ be a pseudo-Anosov braid. Then $\widehat{\beta^k}$ is hyperbolic for $k>4$, provided $\widehat{\beta^k}$ is not one of the following:
	\begin{enumerate}
		\item a torus knot,
		\item a 2-component $T^2 \times I-$link,
		\item a trivial 3-component link.
	\end{enumerate}  
\end{remark}

We refer to \cite{ito} for the notion of Dehornoy floor $[\beta]_D$ of a braid $\beta$.
In \cite[Theorem 2]{ito}, Ito shows the following:
\begin{theorem}\cite{ito}\label{thm-ito}
Let $\beta \in B_m$ be a braid whose closure is a knot satisfying the condition that its Dehornoy floor $[\beta]_D$ satisfies $[\beta]_D\geq 3$. Then $\widehat{\beta}$ is hyperbolic if and only if $\beta$ is pseudo-Anosov.
\end{theorem}

\begin{remark}[Genericity] \label{rmk-generic} In \cite{caruso1,caruso2}, Caruso and Wiest proved the genericity of  pseudo-Anosov  elements of braid groups.
	More precisely, equip the braid group $B_m$ ($m\geq 3$) with Garside’s generating set, and let $\Gamma$ be the resulting Cayley graph. Then it is shown in \cite{caruso2} that the proportion of pseudo-Anosov braids in the ball of radius $r$ in $\Gamma$ tends to $1$ exponentially fast as $r\to \infty$. 
	
	 Remark~\ref{rmk-guaranteeingpahyp} and Theorem~\ref{thm-ito} now furnish two different ways of upgrading a pseudo-Anosov braid $\beta$ to a hyperbolic link  in $S^3$.
\end{remark}

\section{Proof of Theorem \ref{thm-main}}\label{sec-pfmain}

\subsection{Construction of the manifolds}\label{subsec-mfds}
We begin by constructing the required $3$-manifolds.   Let $\beta \in B_m$ 
be a braid on $m$ strands. Let $c_+$ and $c_-$ respectively denote the number of positive and negative crossings of $\beta$. 
\begin{hyp}\label{hyp-braids}
Let us suppose $\beta \in B_m$ is a braid such that,
\begin{enumerate}
	\item The closure $\widehat{\beta}$ is a hyperbolic \emph{knot}, 
	\item $c_+ - 2 c_- - m \geq 1$,
	{\item $c_+ + c_- \equiv m + 1 \pmod{2}$.}
\end{enumerate}
\end{hyp}

\begin{remark} {The third condition of Hypothesis~\ref{hyp-braids} is only logically necessary in Section \ref{subsec-hprincfail} to guarantee that the non-contactomorphic tight contact structures on $M_n$ we will construct in Section \ref{subsec-cntstr} are homotopic as $2$-plane distributions. In particular, the first two conditions of Hypothesis~\ref{hyp-braids} are sufficient to ensure all but the third conclusion in Theorem \ref{thm-main}.}
\end{remark}

\noindent {\bf Satisfying Hypothesis~\ref{hyp-braids}:}\\
{Theorem~\ref{thm-incgenus} and Corollary~\ref{cor-pahakengenus} together furnish a large family of examples of 
	 $\beta \in B_m$  such that $\widehat{\beta}$ is a hyperbolic knot using the notion of Haken genus (see Remark~\ref{rmk-guaranteeingpahyp}). Theorem~\ref{thm-ito} uses the notion of Dehornoy floor to give a different sufficient condition.
Hence, braids satisfying the first condition above  are furnished by the discussion at the end of Section \ref{sec-braid}. In particular, Remark~\ref{rmk-generic} shows that such $\widehat{\beta}$ may be obtained from generic braids.} 

The second condition is much easier to satisfy, once a pseudo-Anosov braid $\beta$ has been fixed. Let $\Delta$ denote the Garside element in the braid group $B_m$ given by 
	$$\Delta:=(\sigma_1\sigma_2 \cdots \sigma_{m-1})(\sigma_1\sigma_2\cdots \sigma_{m-2})\cdots  (\sigma_1\sigma_2)(\sigma_1).$$  Then $\Delta^2$ generates the infinite cyclic center of $B_m$ and has only \emph{positive crossings}. If $\beta$ is pseudo-Anosov, then so is $\Delta^{2\ell}\beta$ for all $\ell \in \Z$. Next, let 
	$\widehat{\Delta^{2\ell}\beta}$ denote the closure of  $\Delta^{2\ell}\beta$ in $S^3$. Then
$S^3 \setminus \widehat{\Delta^{2\ell}\beta}$ is obtained from $(S^1 \times D^2)\setminus
 \widehat{\beta}$ by attaching a disk to the $(1,\ell)-$curve on the boundary $S^1 \times S^1$ of $S^1 \times D^2$. Choosing $\ell>0$ large enough,
the second condition is also satisfied by Thurston's hyperbolic Dehn surgery \cite[Chapter 4]{thurston-book}. 

{For the third condition, we observe that $c_+ + c_-$ is the total number of crossings (or equivalently, length of the braid word) of $\beta$.  This can be ensured, for instance, (together with the first and third conditions) as follows: let $\beta$ be a braid on $m$ strands such that $m$ is odd, and $\beta^2$ is a knot. That is to say, the permutation corresponding to $\beta$ squares to an $m$-cycle. Then the total number of crossings of $\beta^2$ is even. Hence, it agrees with the parity of $m+1$. Thus, $\Delta^{2\ell} \beta^2$ satisfies all three conditions provided $\ell > 0$ is sufficiently large, as in the previous paragraph.} \\

\noindent {\bf Constructing hyperbolic homology spheres:}\\
Let  $M_n = S^3_{1/n}(\widehat{\beta})$ denote the $1/n$-Dehn filling of $S^3 \setminus \widehat{\beta}$. Here, we use the convention that  $1/n$-Dehn filling attaches a disk to $\mu + n\lambda$, where $\mu$ is the meridian of the knot $\widehat{\beta}$, and $\lambda$ is a longitude given by pushing the knot off along a Seifert surface.
This is an integer homology $3$-sphere, and by Thurston's hyperbolic Dehn surgery theorem \cite[Chapter 4]{thurston-book}, $M_n$ is a closed hyperbolic $3$-manifold for $n > 0$ sufficiently large.\\

We now construct a smooth $4$-dimensional handlebody $(X_n, \partial X_n)$ such that $\partial X_n \cong M_n$. Let $K := \widehat{\beta} \subset S^3$ and let $U \subset S^3$ denote the meridian of $K$. By the slam-dunk operation \cite[Section 5.3., p. 163]{gosti-99}, we may rewrite the surgery description of $M_n$ as a surgery on the link $K \cup U \subset S^3$ with surgery coefficients $0$ and $-n$ on $K$ and $U$ respectively. By treating this integral surgery diagram as a Kirby diagram, we obtain a $4$-dimensional handlebody $(X_n, \partial X_n)$ by attaching a pair of $2$-handles $h^2_K, h^2_U$ to $D^4$ along $K$ and $U$, respectively, with framing as in the surgery coefficients.

\subsection{Construction of the contact structures}\label{subsec-cntstr}
The aim of this section is to prove:
\begin{theorem}\label{thm-weinsteinstr}
Let $(X_n, \partial X_n)$  be the 4-dimensional handlebodies with boundary constructed in Section~\ref{subsec-mfds}. Then  $(X_n, \partial X_n)$ admit Weinstein structures.
\end{theorem}

 Below,  $tb$ denotes the \emph{Thurston-Bennequin number}. 
We shall use the following theorem:

\begin{theorem}\cite{yasha-stein-90, gompf-98}\label{thm-weinkirby} Let $(W, \partial W)$ be a smooth $4$-manifold-with-boundary, with a handlebody decomposition $W = D^4 \cup_i h_i^2$ consisting of only $2$-handles. Suppose that there exists a Legendrian link $L = \cup_i K_i \subset \partial D^4$ in the framed isotopy class of the union of the attaching circles of the $2$-handles $h_i^2$ such that the framing on $K_i$ is $tb(K_i) - 1$ for all $i$. Then $W$ admits the structure of a Weinstein domain. \end{theorem}

\begin{figure}[h]
	\centering
	\includegraphics[scale=0.3]{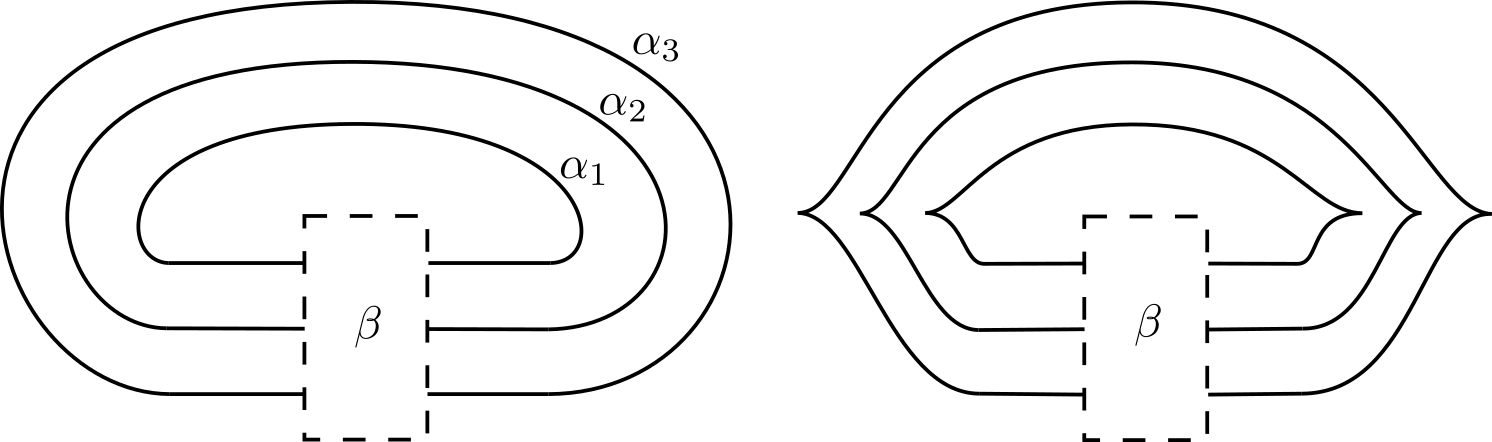}
	\caption{Braid closure $\widehat{\beta}$  obtained by joining the end-points of braid $\beta$ by arcs $\alpha_1, \cdots, \alpha_m$ (left); 
 Legendrian realization of $\widehat \beta$ (right).}
	\label{fig-lag}
\end{figure}

\begin{proof}[Proof of Theorem~\ref{thm-weinsteinstr}]

First, we draw a knot diagram for the braid closure $K = \widehat{\beta}$ (cf. Section \ref{subsec-mfds}) by drawing a horizontal planar diagram of the braid $\beta$, followed by a collection of arcs $\alpha_1, \cdots, \alpha_m$ closing the endpoints of $\beta$. Note that $m$ is the number of strands of $\beta$. We orient $K$ so that the strands of $\beta$ traverse ``left-to-right" while the arcs $\alpha_i$ traverse `` right-to-left". Each arc $\alpha_i$ gives rise to two vertical tangencies of the knot diagram, which we replace by a pair of cusps of appropriate orientation. (See Figure~\ref{fig-lag}.)

{We shall draw every left-handed (or negative) crossing of $\beta$ as in one of the two pictures in Figure~\ref{fig-zigzag}, so that the overcrossing arc has slope less than that of the undercrossing arc. More precisely, let us fix an ordering on all the negative crossings. If $c_{-}$ is even, we draw the negative crossings alternately as in the two pictures in Figure~\ref{fig-zigzag}. If $c_{-}$ is odd, we draw the first negative crossing as in the right picture in Figure~\ref{fig-zigzag}, and alternately use the two  pictures for the rest.}

\begin{figure}[H]
	\centering
	\includegraphics[scale=0.7]{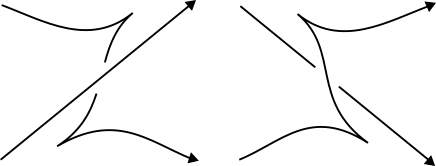}
	\caption{Two {possible} Legendrian realizations of a left-handed crossing by introduction of two cusps of the same orientation.}
	\label{fig-zigzag}
\end{figure}

These modifications lead to a Legendrian front projection for a Legendrian knot $K'$ in the smooth isotopy class of $K \subset S^3$. We note the following  regarding the front diagram for $K'$ that we have constructed:
\begin{enumerate}
\item The number of positive and negative crossings are the same as those for the braid $\beta$, i.e.\ $c_+$ and $c_-$, respectively. 
\item Each arc $\alpha_i$ as well as each negative crossing of $K$ contributes to a pair of cusps. Therefore, the total number of cusps for the front diagram of $K'$ is $2(m + c_{-})$. {The pair of cusps corresponding to each arc $\alpha_i$ are of opposite orientation, while the pair of cusps corresponding to each crossing are of the same orientation.}
\end{enumerate} 
We compute the Thurston-Bennequin number for the Legendrian realization:
\begin{align*}
\mathrm{tb}(K') &= c_{+} - c_{-} - \frac{1}{2} \cdot 2 \left (m + c_{-} \right ) = c_{+} - 2 c_{-} - m.
\end{align*}
{Moreover, by construction, 
\[
\mathrm{rot}(K') =
\begin{cases}
      0, & \text{if $c_{-}$ is even}, \\
      1, & \text{if $c_{-}$ is odd}.
\end{cases}
\]
By Condition $(2)$ of Hypothesis~\ref{hyp-braids} on $\beta$, $\mathrm{tb}(K') \geq 1$. By Condition $(3)$ of Hypothesis~\ref{hyp-braids} on $\beta$, $\mathrm{tb}(K')$ has the same parity as $c_{-}$. Therefore, we may stabilize $K'$ precisely $\mathrm{tb}(K') - 1$ times, and choose the sign of the stabilizations so that $\mathrm{rot}(K') = 0$}. 

Let $K_\mathrm{Leg}$ denote the resulting Legendrian knot obtained by stabilizing $K'$ exactly $\mathrm{tb}(K') - 1$ times in this manner. Then $K_\mathrm{Leg}$ is a Legendrian knot which is still contained in the smooth isotopy class of $K \subset S^3$. Moreover, we have $\mathrm{tb}(K_{\mathrm{Leg}}) = 1$ and $\mathrm{rot}(K_{\mathrm{Leg}}) = 0$.

For any $1 \leq k \leq n-1$, there is  a Legendrian realization $U_{\mathrm{Leg}}$ of the unknot with $\mathrm{tb}(U_{\mathrm{Leg}}) = -n+1$ and $\mathrm{rot}(U_{\mathrm{Leg}}) = 2k - n$. We use such a Legendrian realization for the meridian of $K_{\mathrm{Leg}}$. 

Thus, the framed link $K_{\mathrm{Leg}} \cup U_{\mathrm{Leg}} \subset S^3$ with framing $0$ and $-n$ respectively, satisfies the hypothesis of Theorem \ref{thm-weinkirby}. Since the smooth $4$-dimensional handlebody corresponding to this framed link is the manifold $X_n$ constructed in Section \ref{subsec-mfds}, this produces a Weinstein structure on $X_n$. Note that \emph{a priori} the ambiguity of the choice of $U_{\mathrm{Leg}}$ could lead to different Weinstein structures on $X_n$. We therefore denote the Weinstein domains obtained in this way as $W_{n, k}$, $1 \leq k \leq n-1$. Note that the underlying smooth manifold for $W_{n, k}$ is $X_n$. Thus, the Weinstein structures induce contact structures $\xi_{n, k}$ ($1 \leq k \leq n-1$) on the boundary $M_n = \partial X_n$. Since $\xi_{n, k}$ are Weinstein-fillable contact structures, they must be tight \cite{gromov-85, yasha-90}.
\end{proof}

\subsection{Distinguishing  contact structures}\label{subsec-htpyistpy}

To distinguish the contact structures $\xi_{n, k}$ ($1 \leq k \leq n-1$) on $M_n$ up to isotopy, we follow the overall strategy of Lisca-Mati\'{c} \cite{lismat-97}. We shall use the relative Seiberg-Witten invariants constructed by Kronheimer-Mrowka \cite{kromro-97}. In particular, we need the following theorem:

\begin{theorem}\cite[Theorem 1.2.]{kromro-97}\label{thm-kromro} For a pair $(X, \xi)$ consisting of a compact, connected, oriented $4$-manifold $X$ with nonempty boundary and a contact structure $\xi$ on $\partial X$, there exists a \emph{monopole invariant} 
$$\mathrm{SW} : \mathrm{Spin}^c(X, \xi) \to \Z,$$
defined on the space of $\mathrm{Spin}^c$-structures on $X$ compatible with $\xi$. If $X$ admits an exact symplectic form $\omega$ compatible with $\xi$, then $\mathrm{SW}(\mathfrak{s}) \neq 0$ if and only if $\mathfrak{s} = \mathfrak{s}_0$ is the canonical $\mathrm{Spin}^c$-structure corresponding to $\omega$.\end{theorem}

\begin{proposition}\label{prop-notistpic}
	The contact structures $\xi_{n, k}$, $1 \leq k \leq n-1$, on $M_n$ constructed in Section \ref{subsec-cntstr} are pairwise non-isotopic.\end{proposition}

\begin{proof}
	Suppose $\xi_{n, k}$ and $\xi_{n, l}$ are isotopic contact structures
	on $M_n=\partial X_n$. We choose an isotopy between them, and extend the isotopy into   a symplectized collar of $\partial X_n$ in $X_n$. Thus, we obtain two different Weinstein structures on $X_n$, Weinstein-isotopic to those of $W_{n, k}$ and $W_{n, l}$ respectively, inducing the same contact structure $\xi$ on $M_n$. {Let $\omega_k, \omega_l$ denote the corresponding exact symplectic forms on $X_n$.} Let $\mathfrak{s}_k$ and $\mathfrak{s}_l$ denote the canonical $\mathrm{Spin}^c$ structures on $X_n$ corresponding to the symplectic forms $\omega_k$ and $\omega_l$, respectively. Explicitly, we may think of $\mathfrak{s}_k$ (resp. $\mathfrak{s}_l$) as induced from the choice of an almost complex structure $J_k$ (resp. $J_l$) compatible with $\omega_k$ (resp. $\omega_l$).

By Theorem \ref{thm-kromro}, $\mathrm{SW}(\mathfrak{s}_k)$ and $\mathrm{SW}(\mathfrak{s}_l)$ are both nonzero, and hence
isomorphic. Indeed,  they are respectively equal to the
canonical $\mathrm{Spin}^c$ structures corresponding to possibly different exact symplectic forms $\omega_k$ and $\omega_l$ on $X_n$. However, Theorem \ref{thm-kromro} says that there can only be one $\mathrm{Spin}^c$ structure on which $\mathrm{SW}$ takes nonzero value. In particular, for the two different symplectic forms $\omega_k, \omega_l$, the canonical $\mathrm{Spin}^c$ structures they give rise to are the same.

 Consequently, $J_k, J_l$ are isomorphic almost-complex structures on $X_n$. In particular, $c_1(W_{n, k}, J_k) = c_1(W_{n, l}, J_l)$ as cohomology classes in $H^2(X_n; \Bbb Z)$. Let $[h^2_U] \in H_2(X_n; \Bbb Z)$ denote the homology class corresponding to the $2$-handle $h^2_U$ in $X_n$. As the rotation number of the attaching Legendrian knot of a Weinstein $2$-handle records the first Chern class of the Weinstein handlebody, we have
\begin{align*}\langle c_1(W_{n, k}, J_k), [h^2_U] \rangle &= 2k - n,\\
\langle c_1(W_{n, l}, J_l), [h^2_U] \rangle &= 2l - n.\end{align*}  
Therefore, $k = l$. 
\end{proof}

\begin{proposition}\label{prop-notcntct}There exists a natural number $c > 0$ depending only on the braid $\beta$ chosen in Section \ref{subsec-mfds}, such that for all $n > 0$ sufficiently large, $\{\xi_{n, k} : 1 \leq k \leq n-1\}$ contains a subset of size at least $n/c$ such that all elements of the subset are pairwise non-contactomorphic.\end{proposition}

\begin{proof} By the hypothesis that $n > 0$ is sufficiently large, $M_n$ is hyperbolic by Thurston's Dehn surgery theorem. Suppose $\xi_{n, k}$ and $\xi_{n, l}$ are contactomorphic. As they are not isotopic by Proposition \ref{prop-notistpic}, any contactomorphism between them must be a nontrivial smooth mapping class of $M_n$, i.e., it must not be diffeotopic to identity. By Mostow rigidity and \cite{gameth-03}, the smooth mapping class group of $M_n$ is isomorphic to the isometry group of $M_n$. If $n > 0$ is chosen to be sufficiently large, then $M_n$ admits a unique shortest geodesic given by the core of the  solid torus used in Dehn filling \cite{thurston-book}. Any isometry of $M_n$ must preserve this curve. Deleting this curve results in an isometry of the manifold $S^3 \setminus K$. Therefore, $\mathrm{Isom}(M_n) \cong \mathrm{Isom}(S^3 \setminus K)$. Let $c := |\mathrm{Isom}(S^3 \setminus K)|$ denote the order of this group. This is a finite number by Mostow rigidity, and only depends on the braid $\beta$ whose closure gives the knot $K$. Thus, the action of $\mathrm{MCG}(M_n)$ on the set $\{\xi_{n, 1}, \cdots, \xi_{n, n-1}\}$ has all orbits of size at most $c$. Therefore, there exist at least $n/c$ contact structures, one from each orbit, that are pairwise non-contactomorphic.\end{proof}

\subsection{Homotopy-indistinguishability of plane fields}\label{subsec-hprincfail}

In contrast to Proposition \ref{prop-notistpic} and \ref{prop-notcntct}, we shall show that the contact structures $\xi_{n, k}$ ($1 \leq k \leq n-1$) are all homotopic as oriented rank-$2$ distributions. We will use Gompf's $\theta$-invariant {(also known as the $d_3$-invariant in some of the more recent literature)} to accomplish this.

\begin{definition}\cite[Definition 4.2.]{gompf-98} Let $M^3$ be an oriented $3$-manifold and $\xi$ be a co-oriented plane field on $M$. Let $(X^4, J)$ be an almost-complex $4$-manifold with $\partial X = M$ such that $\xi = TM \cap JTM$. The $\theta$-invariant of $(M, \xi)$ is defined to be
$$\theta(M, \xi) := \langle c_1(X, J)^2, [X, \partial X] \rangle - 2\chi(X) - 3\sigma(X),$$
where $c_1(X, J)$ is the first Chern class of $(X, J)$, $\chi(X)$ is the Euler characteristic of $X$ and $\sigma(X)$ is the signature of $(X, \partial X)$. This is an invariant of $(M^3, \xi)$ independent of the choice of the bounding almost-complex manifold $(X, J)$.
\end{definition}

The following theorem shows that the $\theta$-invariant is a complete invariant of oriented plane fields up to homotopy on integer homology $3$-spheres:

\begin{theorem}\cite{gompf-98}\cite[Proposition 2.2.]{lismat-97} If $M^3$ is an oriented integer homology $3$-sphere and $\xi_1, \xi_2$ are co-oriented plane fields on $M$, then $\xi_1$ is homotopic to $\xi_2$ if and only if $\theta(\xi_1) = \theta(\xi_2)$.
\end{theorem}

\begin{proposition}The contact structures $\xi_{n, k}$, $1 \leq k \leq n-1$, belong to the same homotopy class of oriented rank-$2$ distributions.\end{proposition}

\begin{proof}We show that the $\theta$-invariants of $\xi_{n, k}$ are the same for all $1 \leq k \leq n-1$. We choose the bounding almost-complex $4$-manifold to be $(W_{n, k}, J_k)$. Since the smooth manifold underlying $W_{n, k}$ is just $X_n$, the Euler characteristic and signatures agree for all $1 \leq k \leq n-1$. The homology group $H_2(W_{n, k}; \Bbb Z)$ is freely generated by the fundamental classes $[h^2_K], [h^2_U]$ of the $2$-handles. 

{\begin{enumerate}
\item In Proposition \ref{prop-notistpic}, we had already computed 
$$\langle c_1(W_{n, k}, J_k), [h^2_U] \rangle = 2k - n$$
\item By the construction of $K_{\mathrm{Leg}}$ in Section \ref{subsec-cntstr}, the rotation number of $K_{\mathrm{Leg}}$ is zero. \emph{This is the only place where Condition (3) of Hypothesis~\ref{hyp-braids} is used.} Therefore, $$\langle c_1(W_{n, k}, J_k), [h^2_K] \rangle = 0$$
\end{enumerate}}

Thus, $c_1(W_{n, k}) = (2k - n) \mathrm{PD}[h^2_K]$. Hence, $c_1(W_{n, k})^2 = 0$ for all $1 \leq k \leq n-1$ as $K$ has framing $0$. This concludes the proof.\end{proof}

\section{Further examples: rational surgery on links}\label{sec-furex}

In Section \ref{sec-pfmain}, we constructed tight contact structures on hyperbolic manifolds obtained from choosing appropriate braids $\beta$ with closure $\widehat{\beta}$ a hyperbolic knot, and performing $1/n$-surgery on the closure. In this section, we  generalize  the results of Section \ref{sec-pfmain} in two related directions:
\begin{enumerate}
\item By considering surgery coefficients of the form $p/q$ instead of $1/n$, 
\item By allowing $\widehat{\beta}$  to be a hyperbolic link instead of a hyperbolic knot.
\end{enumerate}
We begin with some notation and preliminary definitions.

\begin{definition}\label{def-negcont}Let $r$ be a real number. A \emph{negative continued fraction expansion} for $r$, denoted $[a_0, a_1, a_2, \cdots]$, is a (possibly non-terminating) expression of the form
$$r = a_0 - \cfrac{1}{a_1 - \cfrac{1}{a_2  - \cdots}}$$
such that $a_i \leq -2$. {These are sometimes known in the literature as Hirzebruch-Jung continued fractions.}
\end{definition}

{
\begin{remark}\label{rmk-negcontrat} We make a few remarks on the existence and termination of negative continued fraction expansions (Definition \ref{def-negcont}).
\begin{enumerate}[leftmargin=*]
\item Every real number $r < -1$ admits a negative continued fraction expansion. We briefly describe an algorithm to produce such a thing. Set $a_0 := -\lceil -r \rceil$. Note $a_0 \leq -2$. If $r$ is an integer, the algorithm terminates and $r = [a_0]$ is the required negative continued fraction expansion. Otherwise, observe:
$$r = a_0 - \frac{1}{(a_0 - r)^{-1}}$$
Now, $a_0 - r = -\lceil -r \rceil - r \in (0, 1)$, as $r$ is not an integer. Thus, $(a_0 - r)^{-1}$ is well-defined and less than $-1$. Reassign $r$ to be $(a_0 - r)^{-1}$ above, and proceed in the same fashion to define $a_1, a_2, \cdots$. Note that we terminate upon encountering an integer in the process. We also have $a_i \leq -2$. Then, $r = [a_0, a_1, a_2, \cdots]$ is a negative continued fraction expansion. 
\item Any real number with a terminating negative continued fraction expansion is rational, but the converse is false. For instance, $-1 = [-2,-2,-2,\cdots]$. \\ \indent
However, if $r < -1$ is a rational number, then the algorithm above does produce a terminating negative continued fraction for $r$. To see this, suppose $r = -m/n < -1$, where $m/n > 0$ and $(m, n) = 1$. The above algorithm can be described as a variant of the Euclidean algorithm. Indeed, find the unique integer $d$ such that $m = n d - k$, where $0 \leq k \leq n-1$. Then $a_0 = -d$. If $k = 0$, the algorithm terminates. If not, we repeat the process with $(a_0 - r)^{-1} = -n/k$. We then find the unique integer $d'$ such that $n = k d' - l$, where $0 \leq l \leq k - 1$. Notice we have a strict inequality $l < k$. This means eventually the algorithm must terminate. Hence, we have produced a terminating continued fraction expansion for $r$.
\end{enumerate}
\end{remark}
}

{
\begin{definition}Let $r < -1$ be a real number. Suppose the algorithm described in Remark \ref{rmk-negcontrat} produces a terminating continued fraction expansion $r = [a_0, a_1, \cdots, a_k]$. Then, we shall denote
$$\Phi(r) := |a_0 + 1| |a_1 + 1| \cdots |a_k + 1|$$
\end{definition}
}

\subsection{Rational surgery on links: {$0 < p/q < 1$}}\label{subsec-surglink}
Let $\beta \in B_m$ be a braid. Let $\sigma : B_m \to S_m$ denote the natural epimorphism from the braid group to the permutation group. The components of the link $\widehat{\beta}$ are in one-to-one correspondence with the cycle decomposition of the permutation $\sigma(\beta) \in S_m$. Suppose $\sigma(\beta)$ has cycle type $(m_1, \cdots, m_\ell)$. Then, we may write $\widehat{\beta}$ as a union of its components, each of which is a braid knot:
$$\widehat{\beta} = \widehat{\beta}_1 \cup \cdots \cup \widehat{\beta}_\ell,$$
where $\beta_i \in B_{m_i}$ is a braid on $m_i$ strands, $1 \leq i \leq \ell$. We fix the following notation:
{
\begin{enumerate}
\item Let $c_{i, +}, c_{i, -}$ be the number of positive and negative crossings of $\beta_i$, respectively.
\item Let $d_{i, j, -}$ be the number of negative crossings between strands of $\beta_i$ and strands of $\beta_j$. We also define,
$$d_{i, -} = \sum_{j \neq i} d_{i, j, -}$$
\end{enumerate}
\begin{remark} While the bookkeeping of $c_{i, -}$ plays a crucial role in ensuring an appropriate Legendrian realization of each component $\widehat{\beta}_i$ of $\widehat{\beta}$, an analogous role is played by the quantity $d_{i, -}$ in ensuring an appropriate Legendrian realization of the crossings between \emph{different} components $\widehat{\beta}_i$ and $\widehat{\beta}_j$ of $\widehat{\beta}$. We shall see this in detail in the proof of Theorem \ref{thm-eglinkrat}.\end{remark}}

\begin{definition}\label{def-slopevec}Let $\beta \in B_m$ be a braid, with its closure $\widehat{\beta} = \widehat{\beta}_1 \cup \cdots \cup \widehat{\beta}_\ell$ as above. Define a \emph{slope vector} to be an ordered tuple of rational numbers,
$$\mathbf{v} = (p_1/q_1, \cdots, p_\ell/q_\ell).$$
Let $M(\beta, \mathbf{v})$  be the $3$-manifold obtained by Dehn filling  $S^3 \setminus \widehat{\beta}$ with surgery coefficient $p_i/q_i$ on the component $\widehat{\beta}_i$,
for all $1 \leq i \leq \ell$. \end{definition}

{\begin{definition} Let $\mathbf{v} = (p_1/q_1, \cdots, p_\ell/q_\ell)$ be a slope vector such that $p_i/q_i \in (0, 1)$ ($1 \leq i \leq \ell)$. By Remark \ref{rmk-negcontrat}, $-q_i/p_i$ has a terminating negative continued fraction expansion. We define
$$\Phi(\mathbf{v}) := \Phi(-q_1/p_1) \Phi(-q_2/p_2) \cdots \Phi(-q_\ell/p_\ell),$$
where $\Phi(r)$ is defined in Definition \ref{def-negcont}.
\end{definition}
}

\begin{theorem}\label{thm-eglinkrat}Let $\beta \in B_m$ be a braid with closure $\widehat{\beta} = \widehat{\beta}_1 \cup \cdots \cup \widehat{\beta}_\ell$ such that {${\beta}_i$ satisfies the condition $c_{i, +} - 2c_{i, -} - d_{i, -} - m_i \geq 1$ for all $1 \leq i \leq \ell$. Let $\mathbf{v} = (p_1/q_1, \cdots, p_\ell/q_\ell)$ be a slope vector such that $0 < p_i/q_i < 1$ for all $1 \leq i \leq \ell$.} Then $M(\beta, \mathbf{v})$ admits at least $\Phi(\mathbf{v})$ distinct tight contact structures up to isotopy. \end{theorem}

\begin{proof}The proof proceeds along the general strategy described in Section \ref{sec-pfmain}. Here we indicate the necessary modifications we have to make. 

\begin{enumerate}[leftmargin=*]
\item As in Section \ref{subsec-mfds}, we construct a {smooth} $4$-manifold-with-boundary $X$ such that $\partial X = M(\beta, \mathbf{v})$. This is accomplished by first writing the {smooth} surgery diagram for $M(\beta, \mathbf{v})$ given by:
$$\widehat{\beta} = \widehat{\beta}_1 \cup \cdots \cup \widehat{\beta}_\ell,$$
with surgery coefficients as the slope vector $\mathbf{v}$. We write $p_i/q_i = 0 - q_i/p_i$ and apply the slam-dunk operation to change surgery coefficients of $\widehat{\beta}_i$ to $0$, at the cost of introducing an {unknotted meridian} for $\widehat{\beta}_i$ with surgery coefficient $-q_i/p_i$, for all $1 \leq i \leq \ell$. {We express $-q_i/p_i$ as a terminating negative continued fraction expansion:
$$-q_i/p_i = [a_{i, 0}, a_{i, 1}, \cdots, a_{i, k_{i}}].$$
We further modify the surgery diagram so that all the coefficients are integral, by iterated applications of slam-dunk. Namely, replace the meridian of each component $\widehat{\beta_i}$ with surgery coefficient $-q_i/p_i$, by a chain of Hopf links with surgery coefficients $a_{i, 0}, a_{i, 1}, \cdots, a_{i, k_i}$ as in the continued fraction expansion.} The final {framed link} defines a Kirby diagram for the desired {smooth} $4$-manifold $X$.

\item {Next, we upgrade the smooth Kirby diagram in Step $(1)$ to a Weinstein-Kirby diagram. This entails choosing a Legendrian representative of the framed link at the end of Step $(1)$, so that the framing condition ``$f = \mathrm{tb} - 1$" in Theorem \ref{thm-weinkirby} holds. \\ \indent
To do this, we first use the technique from the proof of Theorem \ref{thm-weinsteinstr} to choose a Legendrian representative for each component $\widehat{\beta}_i$ with Thurston-Bennequin number $c_{i, +} - 2c_{i, -} - m_i$, for all $1 \leq i \leq \ell$. Then, for any pair of negatively crossing strands between $\beta_i$ and $\beta_j$, we choose a Legendrian realization as in any one of the two pictures in Figure \ref{fig-zigzag}. This procedure introduces a pair of cusps on (the Legendrian realization of) either $\widehat{\beta}_i$ or $\widehat{\beta}_j$. Thus, for each $1 \leq i \leq \ell$, this procedure introduces at most $2d_{i, -}$ cusps on the Legendrian realization of $\widehat{\beta}_i$. Therefore, at the end of the procedure, the Thurston-Bennequin number for $\widehat{\beta}_i$ is at least:
$$(c_{i, +} - 2c_{i, -} - m_i) - \frac{1}{2} \cdot 2d_{i, -} = c_{i, +} - 2c_{i, -} - d_{i, -} - m_i,$$
which is at least $1$ by hypothesis. We apply some number of stabilizations to ensure that the Thurston-Bennequin number is $1$, for all $1 \leq i \leq \ell$.}

\item {Finally, for each of the unknots occurring in the chain of Hopf links described in the framed link in Step $(1)$, we choose a Legendrian realization. This is done so that for any of the unknotted components occurring with framing $f$, the Thurston-Bennequin number is $f + 1$. This is possible, since $f$ occurs as one of the terms in the negative continued fraction expansion of $-q_i/p_i$ for some $1 \leq i \leq \ell$. Therefore, by definition, $f \leq -2$ (see Definition \ref{def-negcont}). \\ \indent
Observe that the possible rotation numbers of a Legendrian unknot with Thurston-Bennequin number $f + 1$ are $2k - f$, $1 \leq k \leq -f - 1$. Thus, there are $|f + 1|$ such choices. As these choices can be made independently on the rest of the components, there are at least $\Phi(\mathbf{v})$ many such choices in constructing the required Legendrian realization of the entire link.}

\item {Step $(2)$ and Step $(3)$ constructs at least $\Phi(\mathbf{v})$ distinct Legendrian realizations of the smooth framed link defining the smooth $4$-manifold $X$ in Step $(1)$. Thus, by Theorem \ref{thm-weinkirby}, we obtain at least $\Phi(\mathbf{v})$ potentially distinct Weinstein structures on $X$. In turn, that gives at least $\Phi(\mathbf{v})$ potentially distinct tight contact structures on $\partial X = M(\beta, \mathbf{v})$. The proof of Proposition \ref{prop-notistpic} now shows that these $\Phi(\mathbf{v})$ many contact structures are indeed pairwise-nonisotopic, since the tuple of rotation numbers of all the components disagree for any two such contact structures. This concludes the proof. \qedhere}

\end{enumerate}
\end{proof}

Let $r < -1$ be an irrational real number. Let $r = [a_0, a_1, \cdots]$ be a continued fraction expansion and let
$$-q_n/p_n = [a_0, a_1, \cdots, a_n]$$
be the $n$-th convergent. Then, we observe $|p_n|+|q_n| \to \infty$, and for infinitely many $i \geq 0$, $a_i \leq -3$. The last assertion follows {from the observation $-1 = [-2, -2, -2, \cdots]$ in Remark \ref{rmk-negcontrat}}, as otherwise $r$ would be rational. Therefore, $\Phi(-q_n/p_n) \to \infty$ as $n \to \infty$. Using this, we deduce the following corollary of Theorem~\ref{thm-eglinkrat}:

\begin{corollary}\label{cor-conthyprat}Let $\beta \in B_m$ be a braid with closure $\widehat{\beta} = \widehat{\beta}_1 \cup \cdots \cup \widehat{\beta}_\ell$ a hyperbolic link.  Suppose that {$\beta_i$ satisfies the condition {$c_{i, +} - 2c_{i, -} - d_{i, -} - m_i \geq 1$} for all $1 \leq i \leq \ell$.} Let $r_1, \cdots, r_{\ell} < -1$ be irrational numbers, and let $-q_{i, n}/p_{i, n}$ denote the $n$-th convergent of the negative continued fraction expansion of $r_i$. Let $\mathbf{v}_n = (p_{1, n}/q_{1, n}, \cdots, p_{\ell, n}/q_{\ell, n})$. Then for sufficiently large $n \geq 0$, $M(\beta, \mathbf{v}_n)$ is a hyperbolic manifold and $\#\mathrm{Tight}(M(\beta, \mathbf{v}_n)) \to \infty$ as $n \to \infty$.
\end{corollary}

\begin{proof}
Since $\widehat{\beta}$ is a hyperbolic link and the surgery coefficients on each component satisfy the condition $|p_{i, n}| + |q_{i, n}| \to \infty$ as $n \to \infty$ for all $1 \leq i \leq \ell$, $M(\beta, \mathbf{v}_n)$ is hyperbolic for sufficiently large values of $n$, by Thurston's hyperbolic Dehn surgery theorem. We have established in Theorem~\ref{thm-eglinkrat} that $M(\beta, \mathbf{v}_n)$ has at least $\Phi(\mathbf{v}_n)$ non-isotopic contact structures. Since $\Phi(\mathbf{v}_n) \to \infty$, the proof of Proposition \ref{prop-notcntct} shows $M(\beta, \mathbf{v}_n)$ admits an arbitrarily large number of pairwise non-contactomorphic contact structures, as $n \to \infty$.
\end{proof}

\subsection{Rational surgery on links: {$p/q > 0$}}\label{subsec-surgknot}
In this section, we extend the results of Section \ref{subsec-surglink} from rational surgeries with coefficients in the interval $(0, 1)$ to \emph{all} positive rational surgeries. 

\subsubsection{The case of knots}\label{subsubsec-surgknot}
For ease of illustration, let us first restrict to the case of knots. We begin with the following key observation.

\begin{proposition}\label{prop-allposcont}Let $K \subset S^3$ be a smooth knot which admits a Legendrian realization $K_{Leg} \subset (S^3, \xi_{std})$ such that $tb(K_{Leg}) = 1$. For any rational number $r > 0$, the three-manifold $S^3_r(K)$ obtained by $r$-surgery along $K$ admits a tight (in fact, Weinstein-fillable) contact structure.  \end{proposition}

\begin{proof}Suppose $r = n + p/q$, where $n$ is a non-negative integer and $p/q \in (0, 1)$. Let 
$$-q/p = [a_0, a_1, \cdots, a_k],$$
be a terminating negative continued fraction expansion, with $a_i \leq -2$ for all $1 \leq i \leq k$. Let $U_0, U_1, \cdots, U_k \subset S^3$ be unknots such that $U_0$ is a meridian of $K$, and $U_i$ is a meridian of $U_{i-1}$ for all $1 \leq i \leq k$. Moreover, let $\mu_1, \mu_2, \cdots, \mu_{2n}$ be $2n$ distinct parallel copies of a meridian of $K$, all distinct from $U_0$ and unlinked from $U_i$ for all $0 \leq i \leq k$. Consider the framed link,
$$L = K \cup U_0 \cup U_1 \cup \cdots \cup U_k \cup \mu_1 \cup \mu_2 \cup \cdots \cup \mu_{2n},$$ 
where $K$ has framing $0$, $U_i$ has framing $a_i$ and $\mu_j$ has framing $-2$. Surgery on each component of $L$ with the prescribed framing coefficients, yields the manifold $S^3_r(K)$. Indeed, this follows from iterated applications of the slam-dunk operation, as 
$$r = n + \frac{p}{q} = 0 - \underbrace{\frac{1}{-2} - \cdots - \frac{1}{-2}}_\text{$2n$ times} - \cfrac{1}{a_0 -\cfrac{1}{\cdots - \cfrac{1}{a_k}}}.$$
It suffices to produce a Legendrian realization of the link $L$ so that the framing of each component satisfies the condition in Theorem \ref{thm-weinkirby}. We achieve this by:
\begin{enumerate}
\item Choosing the Legendrian realization of $K$ to be $K_{\mathrm{Leg}}$, so that $\mathrm{tb}(K_{\mathrm{Leg}}) = 1$,
\item Choosing the Legendrian realization of $U_i$ to be the Legendrian unknot with Thurston-Bennequin number $a_i + 1$.
\item Choosing the Legendrian realization of $\mu_j$ to be the standard Legendrian unknot with Thurston-Bennequin number $-1$.
\end{enumerate}
By Theorem \ref{thm-weinkirby}, this produces a Weinstein $4$-manifold with boundary $S^3_r(K)$. Thus, $S^3_r(K)$ admits a Weinstein-fillable, hence tight, contact structure.
\end{proof}

\begin{remark}\label{rmk-manycontsrat}The proof of Proposition \ref{prop-allposcont} and the discussion in Section \ref{subsec-surglink} shows that $S^3_{n+p/q}(K)$ admits at least $\Phi(-q/p)$ tight  contact structures {that are pairwise non-isotopic}. \end{remark}

\begin{corollary}\label{cor-allposknot} Let $\beta \in B_m$ be a braid satisfying Conditions (1) and (2) of Hypothesis \ref{hyp-braids}. For any rational number $r > 0$, the three-manifold $S^3_r(\widehat{\beta})$ obtained by $r$-surgery along $\widehat{\beta}$ admits a tight contact structure. In particular, for all but finitely many $r > 0$, $S^3_r(\widehat{\beta})$ is a hyperbolic rational homology sphere admitting a tight contact structure.\end{corollary}

\begin{proof}Let $K = \widehat{\beta}$. The proof of Theorem \ref{thm-weinsteinstr} shows that $K$ admits a Legendrian realization with Thurston-Bennequin number $c_{+} - 2c_{-} - m$, which is at least $1$ by Condition $(2)$ of Hypothesis \ref{hyp-braids}. Thus, we may stabilize some number of times to obtain a Legendrian realization with Thurston-Bennequin number $1$. Proposition \ref{prop-allposcont} now shows that all positive rational surgeries on $\widehat{\beta}$ admit tight contact structures. 

Since $\beta$ satisfies Condition $(1)$ of Hypothesis \ref{hyp-braids}, $\widehat{\beta}$ is a hyperbolic knot. Therefore, by Thurston's hyperbolic Dehn surgery theorem, $S^3_r(\widehat{\beta})$ is a hyperbolic manifold for all but finitely many $r > 0$.\end{proof}

\subsubsection{The case of links}
We now indicate the appropriate generalization to the case of links. We begin by observing that the key Proposition \ref{prop-allposcont} generalizes to links as follows.

\begin{lemma}\label{lem-allposcontlink} Let $L = K_1 \cup \cdots \cup K_{\ell} \subset S^3$ be a smooth link. Suppose $L$ admits a Legendrian realization in $(S^3, \xi_{std})$ such that each component $K_i$ has Thurston-Bennequin number $1$. Let $\mathbf{r} = (r_1, \cdots, r_{\ell})$ be a slope vector such that $r_i > 0$. The manifold $S^3_{\mathbf{r}}(L)$ obtained by $r_i$-surgery on $K_i$ (for all $i$), admits a tight contact structure.\end{lemma}

The proof is a straightforward application of the slam-dunk technique used in the proof of Proposition \ref{prop-allposcont}, applied to each component of $L$. 

Next, we specialize Lemma \ref{lem-allposcontlink} to the case of braids $\beta \in B_m$ appearing in Section \ref{subsec-surglink}. Recall that the cycle decomposition of the permutation corresponding to $\beta$ gives rise to braids $\beta_i \in B_{m_i}$ such that $\widehat{\beta}_i \subset S^3$ is a knot, and $\widehat{\beta} = \widehat{\beta}_1 \cup \cdots \cup \widehat{\beta}_\ell \subset S^3$. Let $c_{i, +}, c_{i, -}, d_{i, -}$ be the quantities defined in the beginning of Section \ref{subsec-surglink}. Then we have the following by a similar argument.

\begin{corollary}\label{cor-allposlink}Let $\beta \in B_m$ be a braid with closure $\widehat{\beta} = \widehat{\beta}_1 \cup \cdots \cup \widehat{\beta}_\ell$ a hyperbolic link, such that ${\beta}_i$ satisfies the condition $c_{i, +} - 2c_{i, -} - d_{i, -} - m_i \geq 1$ for all $1 \leq i \leq \ell$. Let $\mathbf{v} = (r_1, \cdots, r_\ell)$ be a slope vector such that $r_i > 0$ for all $1 \leq i \leq \ell$. Then the manifold $M(\beta, \mathbf{v})$ (Definition \ref{def-slopevec}) admits a tight contact structure. In particular, for all but finitely many such slope vectors $\mathbf{v}$, $M(\beta, \mathbf{v})$ is a hyperbolic rational homology sphere admitting a tight contact structure.\end{corollary}

\subsection{Tight contact structures on hyperbolic $L$-spaces}\label{subsec-lspace}

In \cite[Definition 1.1]{os-05} Oszv\'ath and Szab\'o defined an \emph{$L$-space} to be a rational homology three-sphere $M$ for which
$$\mathrm{rk}\,\widehat{HF}(M) = |H_1(M; \Bbb Z)|,$$ 
where $\widehat{HF}$ is the hat flavor of Heegaard-Floer homology. It follows from the works of \cite{os-04, karo-17, bow-16} that $L$-spaces do not admit co-orientable taut foliations. In particular, it is not possible to produce tight contact structures on these manifolds by confoliations \'a la Eliashberg-Thurston \cite{elithu-98}. In this section, we give a recipe to produce examples of hyperbolic $L$-spaces admitting arbitrarily many pairwise non-contactomorphic tight contact structures, as an application of the results in Section \ref{subsec-surgknot}. 

%The $L$-space conjecture \cite{bgw-13} asserts that a rational homology three-sphere admits a taut foliation \emph{if and only if} it is an $L$-space. It was asked by Etnyre \cite[Question 4 and 5]{et-07} whether atoroidal closed three-manifolds possessing universally tight contact structures with infinite fundamental group admit taut foliations. Subsequently, Lekili-Ozbagci \cite{lekozb-10} have produced counterexamples which are small Seifert-fibered $L$-spaces. A revised conjecture has been recently proposed for hyperbolic $L$-spaces by Min-Nonino \cite[Conjecture 1.1]{minnon-23} as a contact-topological analogue of the $L$-space conjecture.

We begin by fixing some notation. Recall from Section \ref{sec-braidclosure} that for a braid $\beta \in B_m$, the closure $\widehat{\beta} \subset S^3$ sits inside a solid torus $S^1 \times D^2 \subset S^3$ that is one of the two handlebodies of a chosen genus $1$ Heegaard decomposition of $S^3$. Let $U \subset D^2 \times S^1$ be the core of the other handlebody. If $\beta$ is pseudo-Anosov, $S^3 \setminus (U \cup \widehat{\beta}) \cong (S^1 \times D^2) \setminus \widehat{\beta}$ is a hyperbolic three-manifold. Recall also the Garside element $\Delta \in B_m$ defined by 
$$\Delta = (\sigma_1 \sigma_2 \cdots \sigma_{m-1}) (\sigma_1 \sigma_2 \cdots \sigma_{m-2}) \cdots (\sigma_1 \sigma_2) \sigma_1.$$
We say a knot $K \subset S^3$ is an \emph{$L$-space knot} if for all sufficiently large rational $r > 0$, the manifold $S^3_r(K)$ obtained by Dehn $r$-surgery on $K$ is an $L$-space.

%Recall that the Garside element $\Delta$ in the braid group $B_m$ is defined to be: $\Delta = (\sigma_1 \sigma_2 \cdots \sigma_{m-1}) (\sigma_1 \sigma_2 \cdots \sigma_{m-2}) \cdots (\sigma_1 \sigma_2) \sigma_1.$

\begin{proposition}\label{prop-lspace}
Let $\beta \in B_m$ be a pseudo-Anosov braid such that $\widehat{\beta}$ is a knot. Let $M$ be the three-manifold given by $(0, k)$-surgery on the link $U \cup \widehat{\beta} \subset S^3$, see Figure \ref{fig-twistknot} (note that $M$ depends on $k$, but we suppress this from the notation for clarity). Suppose 
\begin{enumerate}[label=\normalfont(\arabic*)]
\item $M$ is an $L$-space for all sufficiently large integers $k > 0$,
\item $\Delta^2 \beta$ is an $L$-space knot.
\end{enumerate}
Then for sufficiently large $\ell > 0$, $\widehat{\Delta^{2l} \beta}$ is a hyperbolic $L$-space knot.
\end{proposition}

\begin{figure}[h]
	\centering
	\includegraphics[scale=0.35]{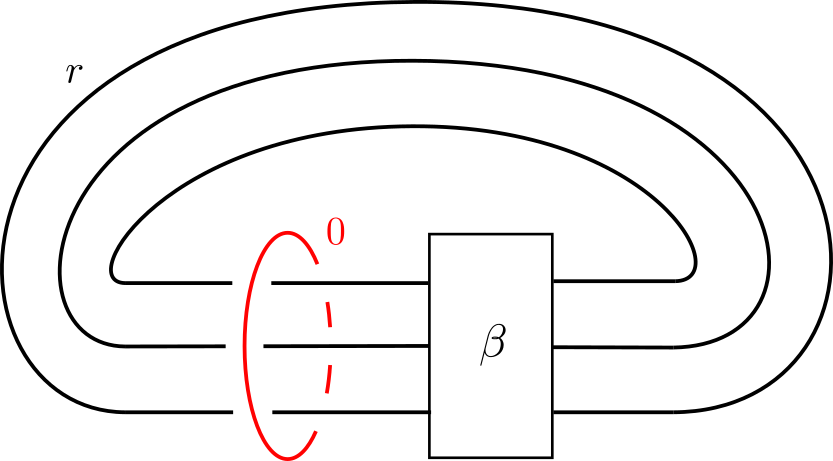}
	\caption{Surgery description of $M$. The red unknotted component is $U$.}
	\label{fig-twistknot}
\end{figure}

\begin{proof}
Let $V \subset S^3 \setminus (U \cup \widehat{\beta})$ be an unknot homotopic to the meridian of $U$, and let $\lambda$ denote the Seifert framing of $V \subset S^3$. The framed knot $(V, \lambda) \subset S^3 \setminus (U \cup \widehat{\beta})$ gives rise to a framed knot in the surgered manifold $M$. By a slight abuse of notation, we also denote this framed knot as $(V, \lambda) \subset M$. Let $\mu$ denote the meridian of $V \subset M$.

Observe that for $\ell \geq 1$, $M_{\lambda + \ell \mu}(V) = S^3_{k + \ell m^2}(\widehat{\Delta^{2\ell} \beta})$. Indeed, the $(\ell, 0, k)$-surgery on the link $V \cup U \cup \widehat{\beta} \subset S^3$ is diffeomorphic to the $(-1/\ell, k)$-surgery on $U \cup \widehat{\beta} \subset S^3$ by the slam-dunk operation. By applying $\ell$ Rolfsen twists \cite[Section 5.3, p. 162]{gosti-99} about $U$, we see that this is in turn equivalent to $(k + \ell m^2)$-surgery along $\widehat{\Delta^{2\ell}\beta} \subset S^3$.

For $\ell \geq 1$, $|H_1(M_{\lambda + \ell \mu}(V); \Bbb Z)| = |H_1(S^3_{k + \ell m^2}(\widehat{\Delta^{2\ell} \beta}); \Bbb Z)| = k + \ell m^2$. Also, 
$$|H_1(M)| = \left | \det \begin{pmatrix}0 & m \\ m & k\end{pmatrix} \right | = m^2.$$
Therefore, we have $|H_1(M)| + |H_1(M_{\lambda + \ell \mu}(V))| = |H_1(M_{\lambda + (\ell+1)\mu}(V)|$. 

By hypothesis, $M$ is an $L$-space. Moreover, $M_{\lambda + \mu}(V) = S^3_{k+m^2}(\widehat{\Delta^2 \beta})$. Since $\Delta^2 \beta$ is an $L$-space knot, $M_{\lambda + \mu}(V)$ is also an $L$-space for sufficiently large $k > 0$. By inductively applying \cite[Proposition 2.1]{os-05}, we conclude that $M_{\lambda + \ell \mu}(V)$ is an $L$-space for all $k, \ell > 0$ sufficiently large. Therefore, $\widehat{\Delta^{2l} \beta}$ is an $L$-space knot for $\ell > 0$ sufficiently large. 

Since $S^3 \setminus \widehat{\Delta^{2l} \beta}$ is obtained from the hyperbolic manifold $S^3 \setminus (U \cup \widehat{\beta})$ by $(-1/\ell)$-Dehn filling of the cusp $U$, it is also hyperbolic for $\ell > 0$ large by Thurston's hyperbolic Dehn surgery theorem \cite[Chapter 4]{thurstonnotes}. Therefore, $\widehat{\Delta^{2l} \beta}$ is a hyperbolic knot for $\ell > 0$ sufficiently large.
\end{proof}

\begin{proposition}\label{prop-lspacetight}Let $\beta \in B_m$ be a braid satisfying all the hypothesis in Proposition \ref{prop-lspace}. Then for sufficiently large integers $\ell, k > 0$,
$$\# \mathrm{Tight}(S^3_{k + 1/n}(\widehat{\Delta^{2\ell} \beta})) \to \infty$$
\end{proposition}

\begin{proof}
For $\ell > 0$ sufficiently large, the braid $\Delta^{2\ell} \beta$ satisfies Hypothesis \ref{hyp-braids}. Therefore, by Section \ref{subsec-surgknot} and Remark \ref{rmk-manycontsrat}, we obtain at least $\Phi(-n) = n-1$ non-isotopic tight contact structures on $S^3_{k + 1/n}(\widehat{\Delta^{2\ell} \beta})$. Since $\widehat{\Delta^{2\ell}\beta}$ is hyperbolic, an argument analogous to Proposition \ref{prop-notcntct} shows that the number of non-contactomorphic tight contact structures is also as large as desired, for $n \to \infty$.
\end{proof}

We give an example to illustrate the existence of braids satisfying the hypothesis of Proposition \ref{prop-lspace}.

\begin{example}\label{eg-lspace}Let $\beta = \sigma_1^{2k+1} \sigma_2^{-1} \in B_3$. Then $\beta$ is a pseudo-Anosov braid, and $\widehat{\beta}$ is the $(2k+1, 2)$-torus knot. First, observe:
$$\Delta^2 \beta = (\sigma _2 \sigma_1^2 \sigma_2) \sigma_1^{2k+5} \sigma_2 (\sigma_2 \sigma_1^2 \sigma_2)^{-1},$$
therefore $\widehat{\Delta^2 \beta}$ is the $(2k + 5, 2)$-torus knot. As torus knots admit lens space surgeries, they are $L$-space knots. Hence Condition $(2)$ of Proposition \ref{prop-lspace} is satisfied.

To verify Condition $(1)$, we use the Montesinos trick \cite{mont-75}. The computation in Figure \ref{fig-montesinos} shows that the manifold given by the surgery diagram in Figure \hyperref[fig-montesinos]{5.(a)} is a branched double cover of a pretzel knot $P(q, 3, -3)$. Observe that $P(q, 3, -3)$ is the Montesinos link $M(1/q, 1/3, -1/3)$ (up to mirror image), which has as its branched double cover the Seifert fibered manifold $M(0; q, 3, -3)$. Here, we use the notation of Lisca-Stipcisz \cite{listi-07}: $M(e_0; r_1, r_2, \cdots, r_n)$ denotes the Seifert fibered manifold obtained by $e_0$-surgery on an unknot having $n$ meridians with surgery coefficients $-1/r_1, -1/r_2, \cdots, -1/r_n$. This admits a Seifert fibration over the sphere with Euler class $e_0$ and $n$ exceptional fibers. 

Therefore, the manifold given by the surgery diagram in Figure \hyperref[fig-montesinos]{5.(a)} must be the Seifert-fibered manifold $M(0; q, 3, -3)$. Since $e_0 = 0$ for $M(0; q, 3, -3)$, \cite[Theorem 1]{listi-07} and \cite{lismat-04} shows it is an $L$-space. This concludes the verification of Condition $(1)$.
\end{example}

\begin{figure}[p]
	\centering
	\includegraphics[scale=0.33]{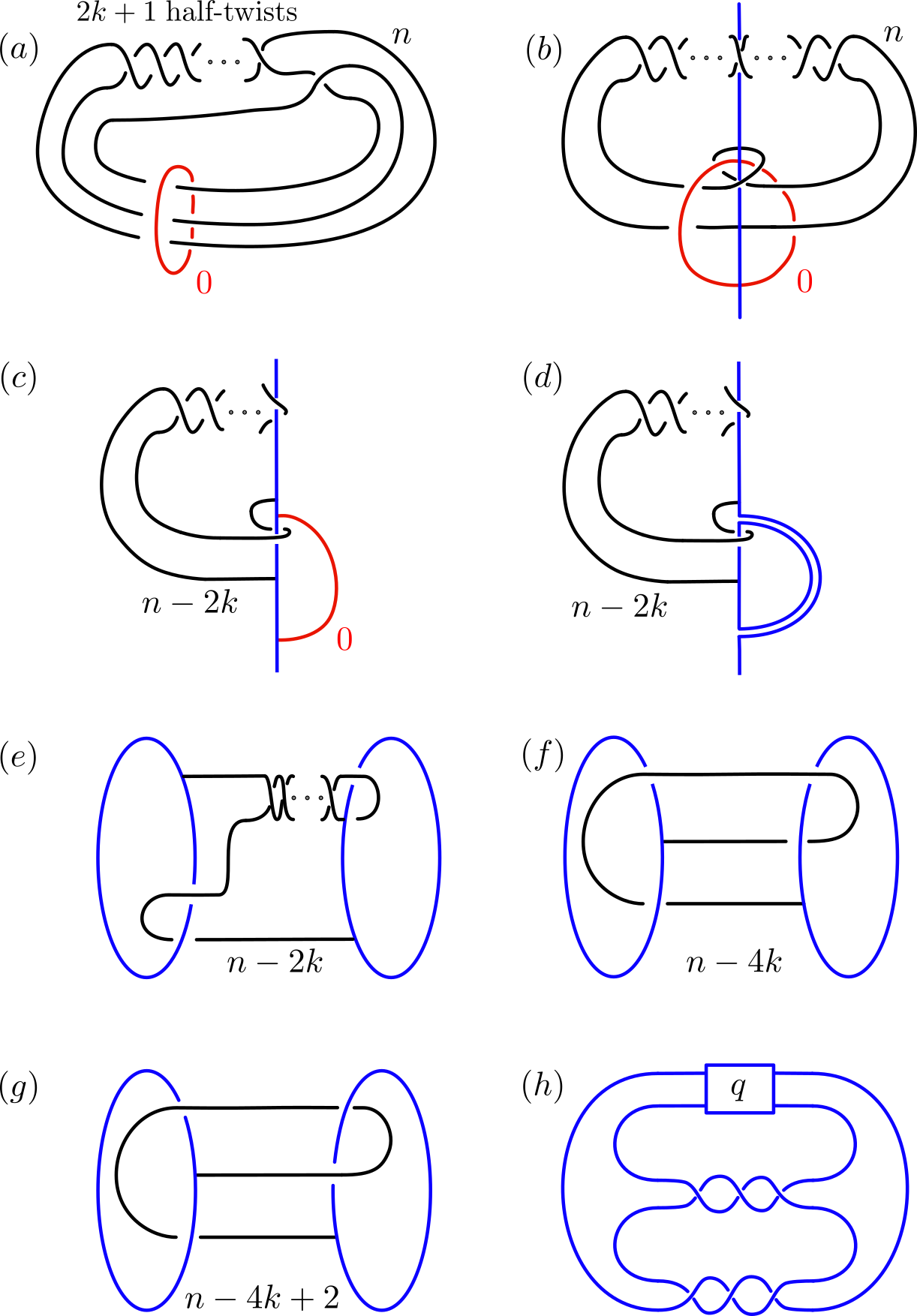}
	\caption{(a) Surgery description of the manifold. (b) Isotopy putting the framed link in a strongly invertible position. (c) Quotient by the $\Bbb Z/2$-symmetry is a rational tangle replacement of the blue unknot along the black and red arcs. The writhe of the black knot in Figure (a) is $2k$, therefore the tangle coefficient of the black arc is $n - 2k$, measured with respect to the blackboard framing of the arc. (d) \& (e) Performing tangle replacement along the red arc and isotopy. (f) \& (g) The black arc can be isotoped to Figure (f) if $k$ is odd and Figure (g) if $k$ is even, with modified tangle coefficient. This is achieved by spinning the endpoint of the black arc counterclockwise along the blue unknotted component on the left. (h) Performing the tangle replacement in Figure (f) \& (g) is tantamount to a band connect sum of two unknots. This produces the pretzel knot $P(q, 3, -3)$ where $q = 4k - 1 - n$ for $k$ odd, and $q = 4k - 2 - n$ for $k$ even.}
	\label{fig-montesinos}
\end{figure}

\begin{remark}
We remark that Min-Nonino \cite{minnon-23} classify tight contact structures for various surgeries on the Whitehead link, which is a hyperbolic $L$-space link. 

Also, note that we could have used any hyperbolic $L$-space knot obtained as closure of a braid $\beta$ satisfying Hypothesis \ref{hyp-braids} to construct hyperbolic $L$-spaces with arbitrarily many distinct tight contact structures. For instance, Baker-Luecke \cite{bl20} construct some intricate examples which are braid positive. In fact, almost all known hyperbolic $L$-space knots are braid positive (and therefore satisfy Hypothesis \ref{hyp-braids} if the braid length is larger than the number of strands), see Baker-Kegel \cite{bl24}. 

That being said, the purpose of Proposition \ref{prop-lspace} is to produce hyperbolic $L$-space knots satisfying Hypothesis \ref{hyp-braids} (in fact, for large $\ell > 0$, $\Delta^{2\ell} \beta$ is always a positive braid) in a somewhat systematic manner. We remark that the strategy behind Proposition \ref{prop-lspace} shares a similarity to the theory of seiferters due to Deruelle-Miyazaki-Motegi \cite{dmm-12}; see also, Baker-Motegi \cite{bakmot-19}. However, it seems unlikely that in general the unknot $U$ in Proposition \ref{prop-lspace} is necessarily a seiferter. 
\end{remark}

\section{Geometric limit of contact structures}\label{sec-geolim}

\subsection{Preliminary definitions}

Recall the following  definition from hyperbolic geometry:

\begin{definition}\label{def-geoconv}Let $M_n$ ($n \geq 1$) and $M$ be finite-volume hyperbolic $3$-manifolds. The sequence of manifolds $(M_n)_{n \geq 1}$ is said to \emph{geometrically converge} to the manifold $M$ if there exists a collection of points $p_n \in M_n$ and $p \in M$, a collection of radii $r_n > 0$, a collection of Lipschitz constants $\lambda_n > 0$ and a collection of $\lambda_n$-Lipschitz diffeomorphisms $\phi_n : B(p_n, r_n) \to B(p, r_n)$ such that $r_n \to \infty$ and $\lambda_n \to 1$ as $n \to \infty$. \end{definition}

We define a notion of geometric convergence for a sequence of hyperbolic $3$-manifolds equipped with a contact structure. 

\begin{definition}\label{def-contgeoconv}Let $(M_n, \xi_n)$ ($n \geq 1$) and $(M, \xi)$ be finite-volume hyperbolic $3$-manifolds equipped with contact structures. We shall say the sequence $(M_n, \xi_n)_{n \geq 1}$ \emph{geometrically converges} to $(M, \xi)$ if, 
\begin{enumerate}
\item $M_n$ converges to $M$ geometrically (cf. Definition \ref{def-geoconv}),
\item We use the notation as in Definition \ref{def-geoconv}. For each $n \geq 1$, there exists a diffeotopy $\{f_t\}_{t \in [0, 1]}$ of $\mathrm{Op}(\partial B_n(p_n, r_n))$ with $f_0 = \mathrm{id}$, such that
$$f_1^*\xi_n = \phi_n^* \xi,$$
and $\{f_t\}$ extends to a diffeotopy $\{\widetilde{f}_t\}_{t \in [0, 1]}$ of $B_n(p_n, r_n)$ satisfying $\widetilde{f}_0 = \mathrm{id}$, $\widetilde{f}_1^*\xi_n = \phi_n^* \xi$. (Here, we are using Gromov's notation of $\mathrm{Op}(Q)$ from \cite{gromov-pdr} to denote a germ of an open neighborhood of $Q$.)
\end{enumerate}
\end{definition}

We remark that geometric limits of contact structures in the sense of Definition \ref{def-contgeoconv} are not necessarily unique. However, the possible ambiguities occur only from contact topological information disappearing to infinity under geometric limits. 

In this section, we investigate geometric convergence of the contact structures constructed in Section \ref{sec-furex}. To do this, we shall require some amount of convex surface theory (see, \cite{honda-1}). We spend some time on this in the next subsection.

\subsection{Basic slice and stabilization}

Let us fix an integral basis of $H_1(T^2; \Bbb Z)$. We recall the following definition due to Honda \cite{honda-1}:

\begin{definition}\cite[Section 4.3]{honda-1} A \emph{basic slice} is a contact structure $\xi^b$ on $T^2 \times [0, 1]$ such that,
\begin{enumerate}
\item $(T^2 \times [0, 1], \xi^b)$ is tight,
\item $T^2 \times \{i\}$, $i = 0, 1$ are convex surfaces, and the number of dividing curves on each is $2$,
\item Let $s_i$ be the slopes of the dividing curves on $T^2 \times \{i\}$, for $i = 0, 1$. If $s_i = p_i/q_i$, $i = 0, 1$ denotes the reduced fraction representing the slopes, then $(p_0, q_0), (p_1, q_1)$ forms an integral basis of $\Bbb Z^2$.
\item $\xi^b$ is minimally twisting, i.e., the dividing curves of every $\partial$-parallel convex torus in $T^2 \times [0, 1]$ have slope between $s_0$ and $s_1$.
\end{enumerate}
\end{definition}

\begin{theorem}\cite[Proposition 4.7]{honda-1}\label{thm-hondaprop} There are exactly two basic slices $(T^2 \times I, \xi^b)$ with $s_0 = 0$ and $s_1 = -1$, up to isotopy relative to boundary. They are distinguished by their relative Euler classes which are $(1, 0)$ and $(-1, 0)$ in $H^1(T^2; \Bbb Z)$, respectively.\end{theorem}

\begin{remark}We shall say that the \emph{sign} of a basic slice is positive (resp. negative), if the relative Euler class referred to in Theorem \ref{thm-hondaprop} is $(1, 0)$ (resp. $(-1, 0)$).\end{remark}

\begin{definition}\cite[Section 4.4.5]{honda-1}\label{def-cfblock} Fix $m \geq 1$. A \emph{continued fraction block of length $m$} is a contact structure $\xi^{cf}_m$ on $T^2 \times [1, m+1]$ such that there exists a decomposition,
$$T^2 \times I = B_1 \cup B_2 \cup \cdots \cup B_{m},$$
where $B_i = T^2 \times [i, i+1]$ for all $1 \leq i \leq m$ and $(B_i, \xi^{cf}_m|_{B_i})$ is a basic slice and the dividing slopes on the boundary tori $T^2 \times \{i\}$ of the basic slices are given by $s_i = -i$ for all $1 \leq i \leq m+1$.
\end{definition}

We shall give a description of basic slices in terms of stabilization of Legendrian knots. Although this description is well-known to experts (see \cite{honda-1} for instance), we give a proof for completeness. First we set up some notation. 

Let $K \subset (S^3, \xi_{\mathrm{std}})$ be a Legendrian knot in the sphere with the standard contact structure. Suppose $\mathrm{tb}(K) = -n < 0$. Then $K$ has a tubular neighborhood $N(K)$ contactomorphic to $(S^1 \times D^2, \xi)$ where 
$$\xi = \ker(\sin(2n\pi z) dx + \cos(2n\pi z) dy)$$
Here, $x, y$ are coordinates of $D^2$ and $z$ is the coordinate on $S^1 = \Bbb R/\Bbb Z$.  We will call $N(K)$ the \emph{standard neighborhood} of $K$. The boundary torus $\partial N(K) \cong T^2$ is a convex surface. The characteristic foliation on $\partial N(K)$ consists of two lines of singularities, called the \emph{Legendrian divides}, and is otherwise foliated by the lines $\{y = \mathrm{const}.\}$, called the \emph{ruling curves}. Let $\mu$ be a meridian curve $\{pt\} \times \partial D^2$ and $\lambda$ be one of the ruling curves on $\partial N(K)$. Let us declare (a little unconventionally) that $\mu$ has slope $0$ and $\lambda$ has slope $\infty$. Then the characteristic foliation has two dividing curves parallel to the Legendrian divides, each of slope $-1/n$. Notice that each ruling curve intersects each dividing curve $n$ times.  (See Figure~\ref{fig-charfol}.)\\

\begin{figure}[h]
	\centering
	\includegraphics[scale=0.5]{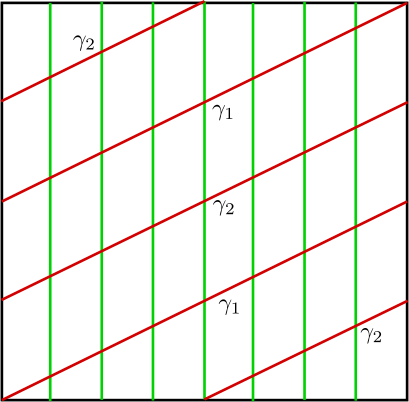}
	\caption{Characteristic foliation of boundary of standard neighborhood of a Legendrian knot. Ruling curves are in green. Legendrian divides are in red.}
	\label{fig-charfol}
\end{figure}

\begin{definition}\label{def-stab}\cite[p. 318]{geiges-book}
	Let $\LL \subset (\R^3,\xi)$ be a Legendrian knot in $\R^3$ equipped with the standard contact structure $\xi$. Let $\pi: \R^3_{xyz} \to \R^2_{xz}$ denote the
	Legendrian front projection. The positive (resp.\ negative) \emph{stabilization}
	of $\LL$ is obtained from $\LL$ by 
	\begin{enumerate}
	\item first choosing a smooth embedded arc 
	$\alpha$  without crossings on $\pi(\LL)$,
	\item modifying the arc 	$\alpha$  in a compactly supported manner by adding two cusps of 
	 positive (resp.\ negative) sign, 
	 \item and finally lifting the modified front projection back to a  Legendrian knot in $(\R^3,\xi)$.
	\end{enumerate}
\end{definition}

\begin{proposition}\label{prop-basicslice}Let $K \subset (S^3, \xi_{\mathrm{std}})$ be a Legendrian knot with $\mathrm{tb}(K) = -n <  0$ and $K'$ denote a stabilization of $K$. We isotope $K'$ to be contained in a standard neighborhood $N(K)$ of $K$. Let $N(K')$ be a standard neighborhood of $K'$ contained in $N(K)$. Let $T = \overline{N(K) \setminus N(K')}$. Then $(T, \xi_{\mathrm{std}}|_T)$ is a basic slice with sign the same as that of the stabilization performed on $K$ to yield $K'$.\end{proposition} 

\begin{proof}Certainly, $(T, \xi_{\mathrm{std}}|_T)$ is tight as it is contactly embedded in the tight contact sphere $(S^3, \xi_{\mathrm{std}})$. Moreover, $\partial T = \partial N(K) \cup \partial N(K')$ is a convex surface, and each torus component has precisely two dividing curves. Note that the reduced fraction representations $-1/n, -1/(n+1)$ of the slopes of the dividing curves of $\partial N(K), \partial N(K')$ (respectively) form an integral basis of $\Bbb Z^2$. Therefore, $(T, \xi_{\mathrm{std}}|_T)$ satisfies conditions $(1), (2)$ and $(3)$ in Theorem \ref{thm-hondaprop}. It remains to show $(T, \xi_{\mathrm{std}}|_T)$ satisfies condition $(4)$, i.e., that it is minimally twisting. 

Let $A \subset T$ be an annulus transverse to $\partial T$, such that $\partial A = c \cup c'$ where $c \subset \partial N(K)$ and $c' \subset \partial N(K')$ are a pair of ruling curves. Then $\partial A \subset T$ is Legendrian. Let $fr$ denote the framing on $\partial A$ induced from $A$. This agrees with the surface framing induced from the tori $\partial N(K)$ and $\partial N(K')$, respectively, as $A$ is transverse to $\partial T$. Since $c$ (resp. $c'$) intersects the dividing multi-curve of $\partial N(K)$ (resp. $\partial N(K')$) at $2n$ (resp. $2n+2$) points, we obtain
\begin{align*}
\mathrm{tw}(c, fr) &= -n,\\
\mathrm{tw}(c', fr) &= -(n+1),
\end{align*}
{where $\mathrm{tw}(c, fr)$ denotes the \emph{twisting number} of the contact planes along $c$, with respect to framing $fr$.} As $\mathrm{tw}(c, fr), \mathrm{tw}(c', fr) < 0$, we may perturb $A$ relative to $\partial A$ to be convex. Since $\mathrm{tw}(c', fr) < \mathrm{tw}(c, fr)$, the Imbalance Principle \cite[Proposition 3.17]{honda-1} implies $A$ contains a bypass attached along the boundary component $c' \subset \partial A$. 

Therefore, $\partial N(K)$ is obtained from attaching a bypass to $\partial N(K')$ along a ruling curve. By a change of basis we may assume that the dividing curves of $\partial N(K), \partial N(K')$ have slopes $-1$ and $0$, respectively. In the new basis, the ruling curve $c' \subset \partial N(K')$ along which the bypass is attached has slope $-(n+1) \in (-\infty, -1)$. Therefore, $(T, \xi_{\mathrm{std}}|_T)$ is minimally twisting \cite{honda-1} (see also, \cite[Theorem 5.11]{etnyre-notes}). This concludes the proof that $(T, \xi_{\mathrm{std}}|_T)$ is a basic slice. 

To determine the sign of the basic slice $(T, \xi_{\mathrm{std}}|_T)$ (in the sense of Theorem \ref{thm-hondaprop}), we compute the relative Euler class $e(\xi_{\mathrm{std}}|_T)$. Recall the convex annulus $A \subset T$ constructed above. Then,
$$\langle e(\xi_{\mathrm{std}}|_T), [A] \rangle = \chi(R_+) - \chi(R_-),$$
where $R_{\pm}$ are the $(\pm)$-ve regions of $A$ in the complement of the dividing arcs of $A$. Such a region is either an annulus, or a disk. If it is an annulus, the Euler characteristic is zero. Each disk region is necessarily bounded by a boundary-parallel arc in $A$, therefore giving rise to a bypass. Since $A$ contains precisely one bypass, we conclude $\langle e(\xi|_{\mathrm{std}}|_T), [A] \rangle = \pm 1$. Whether it is $+1$ or $-1$ depends on whether the unique disk region is a $(+)$-ve region or a $(-)$-ve region. Equivalently, it depends on the sign of the bypass. This is precisely the sign of the stabilization, given by $\mathrm{rot}(K) - \mathrm{rot}(K')$.  \end{proof}

\begin{remark}\label{rmk-cfblocksignord} In Definition \ref{def-cfblock}, there is an ambiguity in the description of the contact structure on the continued fraction block $(T^2 \times [1, m+1], \xi^{cf}_m)$ since each basic slice $B_i$ in the decomposition may have positive or negative sign. However, one can always re-order (or \emph{shuffle}) the basic slices in a continued fraction block by contact isotopies relative to boundary so that $B_1, \cdots, B_k$ are the positive basic slices and $B_{k+1}, \cdots, B_n$ are the negative ones. Indeed, this follows from our description of basic slices in Proposition \ref{prop-basicslice}, as order of positive and negative stabilizations on a Legendrian knot is irrelevant up to isotopy. Therefore, the ambiguity in the description of Definition \ref{def-cfblock} is simply the number $k \in \{0,\cdots, m\}$ of positive basic slices in the block.
\end{remark}

\subsection{Geometric convergence of rational surgeries}

In Section \ref{sec-furex}, we constructed examples of tight contact structures on hyperbolic manifolds obtained from rational surgery on certain hyperbolic braid links. We shall show that appropriate sequences of these geometrically converge to tight contact structures on the link complement, in the sense of Definition \ref{def-contgeoconv}. For ease of exposition, we will only consider braid knots.
%\medskip
%\noindent
%\textit{Setup and notation.} 

\subsubsection{The manifolds}\label{subsubsec-mfds}
Let  $\beta \in B_m$ be a braid such that the closure $K := \widehat{\beta}$ is a hyperbolic knot. Assume $\beta$ satisfies the condition $c_+ - 2c_{-} - m \geq 1$ (cf. Section \ref{subsec-mfds}). Let $r < -1$ be a real number, and $r = [a_0, a_1, a_2, \cdots]$ denote a negative continued fraction expansion (cf. Definition \ref{def-negcont}). Let $-q_n/p_n$ denote the $n$-th convergent of $r$. Let $M_n$ denote the $p_n/q_n$-surgery on $K$. Suppose $|p_n| + |q_n| \to \infty$. For instance, this is the case if $r$ is irrational. Then $M_n$ is a hyperbolic manifold for $n \gg 0$ sufficiently large. 

\subsubsection{The contact structures on the manifolds}\label{subsubsec-cntmfds}
In Theorem~\ref{thm-eglinkrat}, we constructed
$$\Phi(-q_n/p_n) = |a_0 + 1| |a_1 + 1| \cdots |a_n + 1|$$
nonisotopic contact structures on $M_n$. Recall that these contact structures are indexed by tuples $(k_0, k_1, \cdots, k_n)$ where $1 \leq k_i \leq -a_i-1$ for all $0 \leq i \leq n$. $M_n$ equipped with such a contact structure admits a Weinstein filling $W_n$, given by the Weinstein-Kirby surgery presentation
$$K \cup U_0 \cup U_1  \cup \cdots \cup U_n,$$
\begin{enumerate}
\item $K$ is a Legendrian realization of $\widehat{\beta}$ with Thurston-Bennequin number $1$,
\item $U_i$ is a Legendrian unknot with Thurston-Bennequin number $a_i +1$ and rotation number $2k_i - a_i$,
\item $U_0$ is the meridian of $K$,
\item $U_i$ is the meridian of $U_{i-1}$ for all $0 < i \leq n$.
\end{enumerate}
Let $\mathbf{k} = (k_0, k_1, \cdots)$ be an infinite tuple such that $1 \leq k_i \leq |a_i + 1|$. We denote $\mathbf{k}(n) = (k_0, k_1, \cdots, k_n)$. We denote the contact structure on $M_n$ corresponding to the tuple $\mathbf{k}(n)$ constructed above as $\xi_{r, \mathbf{k}(n)}$. 

\subsubsection{The contact structures on the limit manifold}\label{subsubsec-limcntmfd}
Let $M_\infty := S^3 \setminus K$. Let $N(K) \subset S^3$ denote a Legendrian standard neighborhood of $K$. Let $\mu, \lambda$ denote the meridian and longitude of $K$, respectively. Let us fix an integral basis of $H_1(T^2; \Bbb Z)$. We write:
$$M_\infty = S^3 \setminus N(K) \cup T^2 \times [-1, 0] \cup T^2 \times [0, 1] \cup T^2 \times [1, 2] \cup \cdots,$$
where for all $i \geq 0$, $T^2 \times [i, i+1]$ is glued to $T^2 \times [i-1, i]$ along $T^2 \times \{i\}$ by the linear diffeomorphism,
\begin{gather*}T^2 \times [i, i+1] \supset T^2 \times \{i\} \stackrel{\phi_i}{\longrightarrow } T^2 \times \{i\} \subset T^2 \times [i-1, i]\\
\phi_i = \begin{pmatrix}0 & -1 \\ 1 & -a_i\end{pmatrix},
\end{gather*}
and $T^2 \times \{-1\}$ is glued to $\partial N(K)$ by sending $(1,0)$ to $-\lambda$ and $(0, 1)$ to $\mu$. We define a contact structure $\xi_{r, \mathbf{k}}$ on $M_\infty$ such that 
\begin{enumerate}
\item $\xi_{r, \mathbf{k}}$ agrees with $\xi_{\mathrm{std}}$ on $S^3 \setminus N(K)$, where $\xi_{\mathrm{std}}$ is the standard contact structure on $S^3$,
\item For all $i \geq 0$, $T^2 \times [i-1, i]$ is a continued fraction block of length $|a_i + 2|$ (cf. Definition \ref{def-cfblock}), consisting of $k_i-1$ positive basic slices appearing in any order (cf. Remark \ref{rmk-cfblocksignord}).
\end{enumerate}
This can be achieved since the slopes of the dividing curves of consecutive continued fraction blocks match upon applying the diffeomorphism $\varphi_i$. Hence, the contact structures glue consistently. 

\subsubsection{Geometric convergence}\label{subsubsec-gmconv}
We are now prepared to state the main result:

\begin{theorem}\label{thm-geomlim}Let $r < -1$ be a real number with a negative continued fraction expansion $r = [a_0, a_1, a_2, \cdots]$, and let $-q_n/p_n$ denote the convergents. Suppose $|p_n| + |q_n| \to \infty$ as $n \to \infty$. Let $\mathbf{k} = (k_0, k_1, k_2, \cdots)$ be an infinite tuple such that $1 \leq k_i \leq |a_i + 1|$. 
	Let $\mathbf{k}(n)= (k_0, k_1, k_2, \cdots, k_n)$, and $\xi_{r, \mathbf{k}(n)}$ denote the contact structure on $M_n$ constructed in Section~\ref{subsubsec-cntmfds}.
	Then, $(M_n, \xi_{r, \mathbf{k}(n)})$ geometrically converges to $(M_\infty, \xi_{r, \mathbf{k}})$.
\end{theorem}

\begin{proof}Recall that $M_n$ is obtained by a $p_n/q_n$-Dehn surgery on the hyperbolic knot $K := \widehat{\beta}$, and $M_\infty = S^3 \setminus K$ is the knot complement. Let $\varepsilon > 0$ be sufficiently small. Then for all $n$ sufficiently large, $M_n$ admits a $\varepsilon$-thick-thin decomposition,
$$M_n = \mathrm{Thick}_\varepsilon(M_n) \cup \mathrm{Thin}_\varepsilon(M_n),$$
such that $\mathrm{Thin}_\varepsilon(M_n)$ is a tubular neighborhood of the systole of $M_n$ given by core of the surgery solid torus, i.e., a $\varepsilon$-thin Margulis tube. For any given $\varepsilon > 0$, there exists $N(\varepsilon) \gg 0$ such that for all $n \geq N(\varepsilon)$, there exists a $(1 + o(\varepsilon))$-Lipschitz diffeomorphism,
$$\phi_{n, \varepsilon} : \mathrm{Thick}_\varepsilon(M_n) \to \mathrm{Thick}_\varepsilon(M_\infty).$$
Thus, the sequence of hyperbolic manifolds $M_n$ geometrically converges to $M_\infty$.

{For convenience of the reader we recall some relevant terminology. A \emph{Weinstein 2-handle} is the standard $D^2 \times D^2 \subset \R^4$ equipped with 
\begin{enumerate}
\item the  standard symplectic form $\omega_{std} = dx_1\wedge dy_1 + dx_2\wedge dy_2$ induced from $\R^4$,
\item the Liouville vector field $$X=-x_1\frac{\partial}{\partial x_1}-x_2\frac{\partial}{\partial x_2}+2y_1\frac{\partial}{\partial y_1}+2y_2\frac{\partial}{\partial y_2}.$$
\end{enumerate}
The regions $S^1 \times D^2$ and  $  D^2 \times S^1$ are called the \emph{attaching}
and \emph{co-attaching regions} respectively. The circles $S^1 \times \{0\}$ and  $  \{0\} \times S^1$ are called the \emph{core circle}
and \emph{belt circle} respectively. Since $(D^2 \times D^2, \omega_{std}, X) $ is a Weinstein manifold-with-corners, the attaching and co-attaching regions $S^1 \times D^2$ and  $  D^2 \times S^1$ inherit canonical induced contact structures.
}

Recall the Weinstein filling $W_n$ of $M_n$ described in Section \ref{subsubsec-cntmfds}, given by the Weinstein-Kirby presentation $K \cup U_0 \cup U_1  \cup \cdots \cup U_n$. Let $h^2_K$ and $h^2_{U_i}$ ($0 \leq i \leq n$) denote the Weinstein $2$-handles in $W_n$ corresponding to the components of this link. We smoothly slide the $2$-handle $h^2_{U_n}$ over the boundary $h^2_{U_{n-1}}$ so that the attaching circle of $h^2_{U_n}$ is the belt circle of $h^2_{U_{n-1}}$. The effect on the boundary of the $4$-manifold is the same as the slam-dunk isotopy \cite[Section 5.3., p. 163]{gosti-99} pushing $U_n$ inside the surgery solid torus of $U_{n-1}$. Similarly, we inductively slide $h^2_{U_i}$ over $h^2_{U_{i-1}}$ for all $1 \leq i \leq n$, and finally slide $h^2_{U_0}$ over $h^2_K$. Taking boundary, we obtain the decomposition:
\begin{equation}\label{eqn-neck}M_n \cong (S^3 \setminus N(K)) \cup T^2 \times [-1, 0] \cup T^2 \times [0, 1] \cup \cdots \cup T^2 \times [n-1, n] \cup S^1 \times D^2\end{equation}
Here,
\begin{enumerate}
\item $N(K)$ is a fixed standard neighborhood of the Legendrian knot $K \subset (S^3, \xi_{\mathrm{std}})$,
\item $T^2 \times [-1, 0]$ is the co-attaching region of $h^2_K$ with a neighborhood of the belt circle removed. Or equivalently, the surgery solid torus for $K$ with a neighborhood of the core circle removed.
\item $T^2 \times [i, i+1]$, $0 \leq i \leq n-1$ is the co-attaching region of $h^2_{U_i}$ after the slide, with a neighborhood of the belt circle removed. Or equivalently, the surgery solid torus of $U_i$ after the slam-dunk isotopy, with a neighborhood of the core circle removed.
\item $S^1 \times D^2$ is the co-attaching region of $h^2_{U_n}$ after the slide. Or equivalently, the surgery solid torus of $U_n$ after the slam-dunk isotopy.
\end{enumerate}

We may imagine the hyperbolic structure on $M_n$ as the hyperbolic manifold-with-boundary $S^3 \setminus N(K)$ and a Margulis tube given by the solid torus in Item $(4)$, connected by a ``long, narrow neck" $T^2 \times [-1, n]$. Observe that pulling back $\xi_{r, \mathbf{k}(n)}$ by the handleslide/slam-dunk isotopies, we get a contact structure which agrees with the contact structure on $S^3 \setminus N(K)$ induced from $(S^3, \xi_{\mathrm{std}})$. Furthermore, the contact structure on the solid torus in Item $(4)$ is the canonical contact structure on the co-attaching boundary of a Weinstein $2$-handle. Thus, convergence of the contact structures $\xi_{r, \mathbf{k}(n)}$ on $M_n$ need only be verified on the neck region.

We proceed to carry this out more precisely. Let $M_\infty = S^3 \setminus K$ be hyperbolic, so that $M_\infty=\Hyp^3/\Gamma$. Then an embedded incompressible torus
in  $M_\infty$ will be called a \emph{horotorus} if it is the quotient of a horosphere by a $\Z\oplus \Z$ stabilizing it.
With respect to the complete hyperbolic structure on $M_\infty = S^3 \setminus K$, fix now a horotorus $T_{-1} \subset M_\infty$. Let $T_n \subset M_\infty$ be a horotorus such that the unbounded component of $M_\infty \setminus T_n$ is exactly the $1/n$-thin part of $M_\infty$. Let $T_0, \cdots, T_{n-1}$ be such that $T_{-1}, T_0, \cdots, T_n$ are a collection of equidistant, parallel horotori. Since $(M_n)$ converges to $M_\infty$, after possibly passing to a subsequence of $(M_n)$ we can find a $(1 + o(1/n))$-Lipschitz diffeomorphism
$$\varphi_n := \phi_{n, 1/n} : \mathrm{Thick}_{1/n}(M_n) \to \mathrm{Thick}_{1/n}(M_\infty)$$
We isotope the contact structure $\xi_{r, \mathbf{k}(n)}$ on $M_n$ so that in Equation \ref{eqn-neck}, the tori $T^2 \times \{i\}$ are precisely $\varphi_n^{-1}(T_i)$, for $-1 \leq i \leq n$.

We equip $M_\infty$ with the contact structure $\xi_{r, \mathbf{k}}$ constructed in Section \ref{subsubsec-limcntmfd}, where the domains bounded by the horotori $T_{i-1}, T_i$ ($i \geq 0$) are continued fraction blocks of length $|a_i + 2|$ containing $k_i - 1$ positive basic slices. By Proposition \ref{prop-basicslice}, $\varphi_n$ is a contactomorphism. Therefore, the diffeomorphisms $\varphi_n$ certify the geometric convergence of $(M_n, \xi_{r, \mathbf{k}(n)})$ to $(M_\infty, \xi_{r, \mathbf{k}})$.
\end{proof}

\subsubsection{Uncountably many tight contact limits} In Tripp \cite{tripp-06}, it is shown that there are uncountably many tight contact structures on an irreducible open three-manifold which are pairwise non-contactomorphic, provided the manifold possesses an end of nonzero genus. In this section we exhibit a family of such examples which also naturally arise as geometric limits of contact structures on closed three-manifolds. We begin with the following notion due to \cite{tripp-06}:

\begin{definition}Two contact structures $\xi, \xi'$ on a three-manifold $M$ are said to be \emph{properly isotopic} if there exists an  isotopy of $M$ taking $\xi$ to $\xi'$. More precisely, there exists a smooth homotopy of diffeomorphisms $\{f_t : M \to M\}$ such that $f_0 = \mathrm{id}$, $f_1^*\xi' = \xi$. \end{definition}

As Gray's stability theorem fails on noncompact manifolds, this is distinct from the notion of an isotopy of contact plane fields in general. Next, we recall the notion of \emph{slope at infinity} introduced in \cite{tripp-06}, which is an invariant of contact structures on open three-manifolds up to proper isotopy. We specialize the definition to the situation of a manifold with toric end, for simplicity. 

\begin{definition}\cite[Section 3]{tripp-06} Let $(M, \xi)$ be a contact three-manifold, where $M$ is the interior of a compact three-manifold $\overline{M}$ with torus boundary $\partial \overline{M} \cong T^2$. Fix an integral basis of $H_1(\partial M; \Bbb Z)$. We say an embedded convex torus $T \subset M$ is \emph{end-parallel} if $T \subset \overline{M}$ is isotopic to $\partial \overline{M}$. Let $T \times [0, \infty) \subset M$ be the unbounded component bounded by such a torus. We call the contact manifold $E = (T \times [0, \infty), \xi)$ together with its embedding in $M$, a \emph{contact end} of $M$. The set of contact ends denoted $\mathrm{Ends}(M, \xi)$ is a directed set, directed by reverse inclusion. 

Fix an integral basis of $H_1(\partial M; \Bbb Z)$. For any end-parallel convex torus $T$, we obtain a corresponding integral basis of $H_1(T; \Bbb Z)$. We denote the slope of the dividing curves of $T$ measured against this basis as $slope(T)$. For any contact end $E$ of $M$, we define the slope of $E$ as 
$$slope(E) = \sup_T slope(T),$$
where $T$ varies over all end-parallel convex tori contained in $E$. Then $slope : \mathrm{Ends}(M, \xi) \to \Bbb R \cup \{\infty\}$ defines a net. If the net is convergent, we define the \emph{slope at infinity} of $M$ to be the limit.
\end{definition}

\begin{definition}Let $r < -1$ be a real number with negative continued fraction expansion $r = [a_0, a_1, a_2, \cdots]$. Let $\mathbf{k} = (k_0, k_1, k_2, \cdots)$ be an infinite tuple such that $1 \leq k_i \leq |a_i + 1|$. 
		We shall say the \emph{sign} of $\mathbf{k}$ is $+$ (resp. $-$) if for all $i \geq 0$ sufficiently large, $k_i = |a_i +1|$ (resp. $k_i = 1$). Otherwise, we say sign of $\mathbf{k}$ is $\pm$.
\end{definition}

We have the following corollary of Theorem \ref{thm-geomlim}:

\begin{corollary}\label{cor-propisopunctbl} Let $\xi_{r,\mathbf{k}}$ denote the contact structures on $M_\infty$ constructed in Section \ref{subsubsec-limcntmfd}. Then,
\begin{enumerate}[label=\normalfont(\arabic*)]
\item $\xi_{r, \mathbf{k}}$ is a tight contact structure on $M_\infty$,
\item $\xi_{r, \mathbf{k}}$ and $\xi_{r', \mathbf{k}'}$ are properly isotopic if and only if $r = r'$ and $\mathbf{k}, \mathbf{k}'$ are of the same sign.
\end{enumerate}
\end{corollary}

\begin{proof}\leavevmode
\begin{enumerate}[label=\normalfont(\arabic*), leftmargin=*]
\item Suppose $(M_\infty, \xi_{r, \mathbf{k}})$ was overtwisted. Then it would contain an overtwisted disk $D_{ot}$. By compactness of the disk, there exists $n \gg 0$ large such that $D_{ot}$ is entirely contained inside $\mathrm{Thick}_{1/n}(M_\infty)$. Recall the diffeomorphism $\varphi_n : \mathrm{Thick}_{1/n}(M_n) \to \mathrm{Thick}_{1/n}(M_\infty)$ from the Proof of Theorem \ref{thm-geomlim}. Thus, $\varphi_n^{-1}(D_{ot}) \subset \mathrm{Thick}_{1/n}(M_n)$ is an overtwisted disk in $\varphi_n^* \xi_{r, \mathbf{k}}$. As $\varphi_n^*\xi_{r, \mathbf{k}}$ is isotopic to $\xi_{r, \mathbf{k}(n)}$, we conclude $(M_n, \xi_{r, \mathbf{k}(n)})$ is also overtwisted. This is a contradiction.

\item As $\xi_{r, \mathbf{k}}$ is tight, by \cite[Proposition 4.1]{tripp-06} the slope at infinity exists. By taking the cofinal sequence of ends bounded by $T^2 \times \{i\}$ in Section \ref{subsubsec-limcntmfd}, we see that the slope at infinity is $r$ (after choosing an appropriate reference homology basis for the toric end). Finally, observe that if $\mathbf{k}, \mathbf{k}'$ are of the same sign, then an isotopy between $\xi_{r, \mathbf{k}}, \xi_{r, \mathbf{k}'}$ can be constructed by shuffling the positive and negative basic blocks (cf. Remark \ref{rmk-cfblocksignord}) \qedhere
\end{enumerate}
\end{proof}

The following Corollary shows that the notion of geometric convergence of contact structures in Definition~\ref{def-contgeoconv} is compatible with Tripp's work in \cite{tripp-06}.
\begin{corollary} \label{cor-inflts} $M_\infty$ admits uncountably many tight contact structures up to contactomorphism arising as geometric limits of tight contact structures on $M_n$.\end{corollary}

\begin{proof}By Corollary \ref{cor-propisopunctbl}, $M_\infty$ admits uncountably many tight contact structures up to proper isotopy. Note that $M_\infty = S^3 \setminus K$ is a knot complement. Observe that the smooth mapping class group of the open manifold $M_\infty$ agrees with the smooth mapping class group (not relative to the boundary) of the thickened knot complement $S^3 \setminus \nu(K)$. Indeed, any diffeomorphism of $M_\infty$ must send an end-parallel torus $T^2 \subset M_\infty$ to an isotopic copy of itself. After isotopy, we may ensure the diffeomorphism preserves this torus. We cut along this torus and observe that the smooth mapping class group $\mathrm{MCG}_{\partial}(T^2 \times [0, \infty))$, relative to $T^2 \times \{0\}$, is trivial. 

Since the thickened knot complement is oriented Haken, a theorem of Waldhausen (see, for instance, Hempel \cite[Corollary 13.7]{hempel-book}) implies the smooth mapping class group $\mathrm{MCG}(S^3 \setminus \nu(K))$ is isomorphic to the group of self-homotopy equivalences $\mathrm{hAut}(S^3 \setminus \nu(K))$. Since $K$ is hyperbolic, the latter is a finite group by the Mostow rigidity theorem. Therefore, $\mathrm{MCG}(M_\infty)$ is finite. 

Thus, we must have uncountably many mapping class orbits of tight contact structures, as desired. \end{proof}

%Since the knot complements we consider are hyperbolic, the smooth mapping class group $\mathrm{MCG}(M_\infty)$ is finite (by Mostow rigidity and \cite{gameth-03} for instance). 

\bibliography{contactbibs}
\bibliographystyle{alpha}
\end{document}